\documentclass{amsart}[12pt]
\copyrightinfo{2006}{American Mathematical Society}
\usepackage{mathrsfs,amscd}

\newtheorem{theorem}{Theorem}[section]
\newtheorem{lemma}[theorem]{Lemma}

\theoremstyle{definition}
\newtheorem{definition}[theorem]{Definition}
\newtheorem{example}[theorem]{Example}

\newtheorem{proposition}[theorem]{Proposition}

\theoremstyle{remark}
\newtheorem{remark}[theorem]{Remark}

\numberwithin{equation}{section}

\begin{document}

\title[]{A new kind of Lax-Oleinik type operator with parameters for time-periodic positive definite Lagrangian systems}

\author[K. Wang and J. Yan]{Kaizhi Wang$^{1,\,2}$ and Jun Yan$^{1}$}
\address{$^{1}$ School of Mathematical Sciences
and Key Lab of Mathematics for Nonlinear Science, Fudan
University, Shanghai 200433, China}

\address{$^{2}$ College of Mathematics,
Jilin University,  Changchun 130012, China}

\email{kaizhiwang@163.com; \ yanjun@fudan.edu.cn}

\date{December 2010}

\begin{abstract}
In this paper we introduce a new kind of Lax-Oleinik type operator
with parameters associated with positive definite Lagrangian
systems for both the time-periodic case and the time-independent
case. On one hand, the new family of Lax-Oleinik type operators
with an arbitrary $u\in C(M,\mathbb{R}^1)$ as initial condition
converges to a backward weak KAM solution in the time-periodic
case, while it was shown by Fathi and Mather that there is no such
convergence of the Lax-Oleinik semigroup. On the other hand, the
new family of Lax-Oleinik type operators with an arbitrary $u\in
C(M,\mathbb{R}^1)$ as initial condition converges to a backward
weak KAM solution faster than the Lax-Oleinik semigroup in the
time-independent case.
\end{abstract}

\maketitle

\section{Introduction}
Let $M$ be a compact and connected smooth manifold. Denote by $TM$
its tangent bundle and $T^*M$ the cotangent one. Consider a
$C^\infty$ Lagrangian $L: TM\times\mathbb{R}^1\to\mathbb{R}^1$,
$(x,v,t)\mapsto L(x,v,t)$. We suppose that $L$ satisfies the
following conditions introduced by Mather \cite{Mat91}:

\vskip0.2cm
\begin{itemize}
    \item [(H1)] \textbf{Periodicity}. $L$ is 1-periodic in the
                 $\mathbb{R}^1$ factor, i.e.,
                 $L(x,v,t)=L(x,v,t+1)$ for all $(x,v,t)\in TM\times\mathbb{R}^1$.
    \item [(H2)] \textbf{Positive Definiteness}. For each $x\in M$ and each
                 $t\in\mathbb{R}^1$, the restriction of $L$ to $T_xM\times
                 t$ is strictly convex in the sense that its
                 Hessian second derivative is everywhere positive
                 definite.
    \item [(H3)] \textbf{Superlinear Growth}.
                 $\lim_{\|v\|_x\to+\infty}\frac{L(x,v,t)}{\|v\|_x}=+\infty$
                 uniformly on $x\in M$, $t\in\mathbb{R}^1$,
                 where $\|\cdot\|_x$ denotes the norm induced by a
                 Riemannian metric on $T_xM$. By the compactness
                 of $M$, this condition is
                 independent of the choice of the Riemannian
                 metric.
    \item [(H4)] \textbf{Completeness of the Euler-Lagrange Flow}.
                 The maximal solutions of the Euler-Lagrange
                 equation, which in local coordinates is:
                 \[
                 \frac{d}{dt}\frac{\partial L}{\partial
                 v}(x,\dot{x},t)=\frac{\partial L}{\partial
                 x}(x,\dot{x},t),
                 \]
                 are defined on all of $\mathbb{R}^1$.
\end{itemize}

\vskip0.2cm

The Euler-Lagrange equation is a second order periodic
differential equation on $M$ and generates a flow of
diffeomorphisms $\phi^L_t:TM\times\mathbb{S}^1\to
TM\times\mathbb{S}^1$, $t\in\mathbb{R}^1$, where $\mathbb{S}^1$
denotes the circle $\mathbb{R}^1/\mathbb{Z}$, defined by

\[
\phi^L_t(x_0,v_0,t_0)=(x(t+t_0),\dot{x}(t+t_0),(t+t_0)\
\mathrm{mod}\ 1),
\]
where $x:\mathbb{R}^1\to M$ is the maximal solution of the
Euler-Lagrange equation with initial conditions $x(t_0)=x_0$,
$\dot{x}(t_0)=v_0$. The completeness and periodicity conditions
grant that this correctly defines a flow on
$TM\times\mathbb{S}^1$.

We can associate with $L$ a Hamiltonian, as a function on
$T^*M\times\mathbb{R}^1$: $H(x,p,t)=\sup_{v\in T_xM}\{\langle
p,v\rangle_x-L(x,v,t)\}$, where $\langle \cdot,\cdot\rangle_x$
represents the canonical pairing between the tangent and cotangent
space. The corresponding Hamilton-Jacobi equation is

\begin{align}\label{1-1}
u_t+H(x,u_x,t)=c(L),
\end{align}
where $c(L)$ is the Ma$\mathrm{\tilde{n}}\mathrm{\acute{e}}$
critical value \cite{Man97} of the Lagrangian $L$. In terms of
Mather's $\alpha$ function $c(L)=\alpha(0)$.

In this paper we also consider time-independent Lagrangians on
$M$. Let $L_a:TM\to\mathbb{R}^1$, $(x,v)\mapsto L_a(x,v)$ be a
$C^2$ Lagrangian satisfying the following two conditions:

\vskip0.2cm
\begin{itemize}
      \item [(H2')] \textbf{Positive Definiteness}. For each $(x,v)\in TM$,
                    the Hessian second derivative $\frac{\partial^2L_a}{\partial v^2}(x,v)$
                    is positive definite.
      \item [(H3')] \textbf{Superlinear Growth}.
                    $\lim_{\|v\|_x\to+\infty}\frac{L_a(x,v)}{\|v\|_x}=+\infty$ uniformly on $x\in
                    M$.
\end{itemize}
\vskip0.2cm

Since $M$ is compact, the Euler-Lagrange flow $\phi^{L_a}_t$ is
complete under the assumptions (H2') and (H3').

For $x\in M$, $p\in T_x^*M$, the conjugated Hamiltonian $H_a$ of
$L_a$ is defined by: $H_a(x,p)=\sup_{v\in T_xM}\{\langle
p,v\rangle_x-L(x,v)\}$. The corresponding Hamilton-Jacobi equation
is

\begin{align}\label{1-2}
H_a(x,u_x)=c(L_a).
\end{align}

The Lax-Oleinik semigroup (hereinafter referred to as L-O
semigroup) \cite{Hop,Lax,Ole} is well known in several domains,
such as PDE, Optimization and Control Theory, Calculus of
Variations and Dynamical Systems. In particular, it plays an
essential role in the weak KAM theory (see
\cite{Fat1,Fat97b,Fat98a,Fat4} or \cite{Fat-b}).

Let us first recall the definitions of the L-O semigroups
associated with $L_a$ (time-independent case) and $L$
(time-periodic case), respectively. For each $u\in
C(M,\mathbb{R}^1)$ and each $t\geq0$, let

\begin{align}\label{1-3}
T^a_tu(x)=\inf_\gamma\Big\{u(\gamma(0))+\int_0^tL_a(\gamma(s),\dot{\gamma}(s))ds\Big\}
\end{align}
for all $x\in M$, and

\begin{align}\label{1-4}
T_tu(x)=\inf_\gamma\Big\{u(\gamma(0))+\int_0^tL(\gamma(s),\dot{\gamma}(s),s)ds\Big\}
\end{align}
for all $x\in M$, where the infimums are taken among the
continuous and piecewise $C^1$ paths $\gamma:[0,t]\to M$ with
$\gamma(t)=x$. In view of (\ref{1-3}) and (\ref{1-4}), for each
$t\geq 0$, $T_t^a$ and $T_t$ are operators from
$C(M,\mathbb{R}^1)$ to itself. It is not difficult to check that
$\{T^a_t\}_{t\geq 0}$ and $\{T_n\}_{n\in\mathbb{N}}$ are
one-parameter semigroups of operators, which means $T^a_0=I$ (unit
operator), $T^a_{t+s}=T^a_t\circ T^a_s$, $\forall t,\ s\geq 0$,
and $T_0=I$, $T_{n+m}=T_n\circ T_m$, $\forall n,\ m\in\mathbb{N}$,
where $\mathbb{N}=\{0,1,2,\cdots\}$. $\{T^a_t\}_{t\geq 0}$ and
$\{T_n\}_{n\in\mathbb{N}}$ are called the L-O semigroup associated
with $L_a$ and $L$, respectively.

The L-O semigroup is used to obtain backward weak KAM solutions
(viscosity solutions) first by Lions, Papanicolaou and Varadhan
\cite{LPV} on the $n$-torus $\mathbb{T}^n$ and later by Fathi
\cite{Fat1} for arbitrary compact manifolds. More precisely, for
the time-independent case, Fathi \cite{Fat1} proves that there
exists a unique $c_0\in \mathbb{R}^1$ ($c_0=c(L_a)$), such that
the semigroup $\hat{T}^a_t:u\to T^a_tu+c_0t$, $t\geq 0$ has a
fixed point $u^*\in C(M,\mathbb{R}^1)$ and that any fixed point is
a backward weak KAM solution of (\ref{1-2}). In the particular
case $M=\mathbb{T}^n$, the backward weak KAM solution obtained by
Fathi is just the viscosity solution obtained earlier by Lions,
Papanicolaou and Varadhan. Moreover, Fathi points out that the
above results for the time-independent case are still correct for
the time-periodic dependent case \cite{Fat-b}. Furthermore, for
the time-independent case, he shows in \cite{Fat4} that for every
$u\in C(M,\mathbb{R}^1)$, the uniform limit
$\lim_{t\to+\infty}\hat{T}^a_tu=\bar{u}$ exists and is a fixed
point of $\{\hat{T}^a_t\}_{t\geq 0}$, i.e., $\bar{u}$ is a
backward weak KAM solution of (\ref{1-2}). In the same paper Fathi
raises the question as to whether the analogous result holds in
the time-periodic case. This would be the convergence of
$T_nu+nc(L)$, $\forall u\in C(M,\mathbb{R}^1)$, as $n\to+\infty$,
$n\in\mathbb{N}$. In view of the relation between $T_n$ and the
Peierls barrier $h$ (see \cite{Mat93} or \cite{Fat5,Ber,Con}), if
the liminf in the definition of the Peierls barrier is not a
limit, then the L-O semigroup in the time-periodic case does not
converge. Fathi and Mather \cite{Fat5} construct examples where
the liminf in the definition of the Peierls barrier is not a
limit, thus answering the above question negatively.

The main aim of the present paper is to introduce a new kind of
Lax-Oleinik type operator with parameters (hereinafter referred to
as new L-O operator) associated with positive definite Lagrangian
systems for both the time-periodic case and the time-independent
case. The new L-O operator associated with the time-independent
Lagrangian is a special case of the one associated with the
time-periodic Lagrangian. We show that

\begin{itemize}
    \item   for the time-periodic Lagrangian $L$, the new family of
            L-O operators with an arbitrary continuous function on $M$ as initial
            condition converges to a backward weak KAM solution of
            (\ref{1-1}).
    \item   for the time-independent Lagrangian $L_a$, the new family of
            L-O operators is a one-parameter semigroup of
            operators, and the new L-O semigroup with an arbitrary continuous function on $M$ as initial
            condition converges to a backward weak KAM solution of
            (\ref{1-2}) faster than the L-O semigroup.
\end{itemize}

Without loss of generality, we will from now on always assume
$c(L_a)=c(L)=0$. We view the unit circle $\mathbb{S}^1$ as the
fundmental domain in $\mathbb{R}^1: [0,1]$ with two endpoints
identified.

We are now in a position to introduce the new L-O operators
mentioned above associated with $L$ and $L_a$, respectively.

\subsection{Time-periodic case}
For each $n\in\mathbb{N}$ and each $u\in C(M,\mathbb{R}^1)$, let

\[
\tilde{T}_nu(x)=\inf_{k\in\mathbb{N} \atop n\leq k\leq
2n}\inf_{\gamma}\Big\{u(\gamma(0))+\int_{0}^k
L(\gamma(s),\dot{\gamma}(s),s)ds\Big\}
\]
for all $x\in M$, where the second infimum is taken among the
continuous and piecewise $C^1$ paths $\gamma:[0,k]\rightarrow M$
with $\gamma(k)=x$. One can easily check that for each
$n\in\mathbb{N}$, $\tilde{T}_n$ is an operator from
$C(M,\mathbb{R}^1)$ to itself, and that
$\{\tilde{T}_n\}_{n\in\mathbb{N}}$ is a semigroup of operators.

\begin{definition}\label{def1}
For each $\tau\in[0,1]$ and each $n\in \mathbb{N}$, let
$\tilde{T}_n^\tau=T_\tau\circ\tilde{T}_n$. Then for each $u\in
C(M,\mathbb{R}^1)$,

\[
\tilde{T}_n^\tau u(x)=(T_\tau\circ\tilde{T}_nu)(x)=
\inf_{k\in\mathbb{N} \atop n\leq k\leq
2n}\inf_{\gamma}\Big\{u(\gamma(0))+\int_{0}^{\tau+k}
L(\gamma(s),\dot{\gamma}(s),s)ds\Big\}
\]
for all $x\in M$, where the second infimum is taken among the
continuous and piecewise $C^1$ paths $\gamma:[0,\tau+k]\rightarrow
M$ with $\gamma(\tau+k)=x$.
\end{definition}

For each $\tau\in[0,1]$ and each $n\in \mathbb{N}$, since
$\tilde{T}_n^\tau=T_\tau\circ\tilde{T}_n$ and $T_\tau$,
$\tilde{T}_n$ are both operators from $C(M,\mathbb{R}^1)$ to
itself, then $\tilde{T}_n^\tau$ is an operator from
$C(M,\mathbb{R}^1)$ to itself too. We also provide an alternative
direct proof of the continuity of $\tilde{T}_n^\tau u$ for each
$u\in C(M,\mathbb{R}^1)$ in Section 3. We call $\tilde{T}_n^\tau$
{\em the new L-O operator associated with $L$}. Note that for
$\tau\in[0,1]\setminus\{0\}$,
$\{\tilde{T}_n^\tau\}_{n\in\mathbb{N}}$ is not a semigroup of
operators, while in the particular case $\tau=0$,
$\{\tilde{T}_n^0\}_{n\in\mathbb{N}}=\{\tilde{T}_n\}_{n\in\mathbb{N}}$
is a semigroup of operators as mentioned above. For each
$n\in\mathbb{N}$ and each $u\in C(M,\mathbb{R}^1)$, let
$U^u_n(x,\tau)=\tilde{T}_n^\tau u(x)$ for all $(x,\tau)\in
M\times[0,1]$. Then $U^u_n$ is a continuous function on
$M\times[0,1]$.

Now we come to the main result:

\begin{theorem}\label{th1}
For each $u\in C(M,\mathbb{R}^1)$, the uniform limit
$\bar{u}=\lim_{n\to+\infty}U^u_n$ exists and

\[
\bar{u}(x,\tau)=\inf_{y\in M}\big(u(y)+h_{0,\tau}(y,x)\big)
\]
for all $(x,\tau)\in M\times\mathbb{S}^1$. Furthermore, $\bar{u}$
is a backward weak KAM solution of the Hamilton-Jacobi equation

\begin{align}\label{1-5}
u_\tau+H(x,u_x,\tau)=0.
\end{align}
\end{theorem}

\begin{remark}
For the definition of the (extended) Peierls barrier $h$, see
\cite{Mat93} or \cite{Fat5,Ber,Con}. For completeness' sake, we
recall the definition in Section 3.
\end{remark}

In addition, we discuss the relation among uniform limits
$\lim_{n\to+\infty}U^u_n$, backward weak KAM solutions and
viscosity solutions of (\ref{1-5}). Let $\bar{u}\in
C(M\times\mathbb{S}^1,\mathbb{R}^1)$. Then the following three
statements are equivalent.

\begin{itemize}
    \item There exists $u\in C(M,\mathbb{R}^1)$ such that the uniform limit
           $\lim_{n\to+\infty}U^u_n=\bar{u}$.
    \item $\bar{u}$ is a backward weak KAM solution of
          (\ref{1-5}).
    \item $\bar{u}$ is a viscosity solution of (\ref{1-5}).
\end{itemize}
See Propositions \ref{pr3-5}, \ref{pr3-6} for details.

\vskip0.2cm

\subsection{Time-independent case}
Just like the time-periodic case, for each $n\in\mathbb{N}$ and
each $u\in C(M,\mathbb{R}^1)$, let

\[
\tilde{T}_n^au(x)=\inf_{k\in\mathbb{N} \atop n\leq k\leq
2n}\inf_{\gamma}\Big\{u(\gamma(0))+\int_{0}^k
L_a(\gamma(s),\dot{\gamma}(s))ds\Big\}
\]
for all $x\in M$, where the second infimum is taken among the
continuous and piecewise $C^1$ paths $\gamma:[0,k]\rightarrow M$
with $\gamma(k)=x$. For each $n\in\mathbb{N}$, $\tilde{T}_n^a$ is
an operator from $C(M,\mathbb{R}^1)$ to itself, and
$\{\tilde{T}_n^a\}_{n\in\mathbb{N}}$ is a semigroup of operators.

For each $\tau\in[0,1]$ and each $n\in \mathbb{N}$, let
$\tilde{T}_n^{a,\tau}=T_\tau^a\circ\tilde{T}_n^a$. Then for each
$u\in C(M,\mathbb{R}^1)$,

\[
\tilde{T}_n^{a,\tau} u(x)=(T_\tau^a\circ\tilde{T}_n^au)(x)=
\inf_{k\in\mathbb{N} \atop n\leq k\leq
2n}\inf_{\gamma}\Big\{u(\gamma(0))+\int_{0}^{\tau+k}
L_a(\gamma(s),\dot{\gamma}(s))ds\Big\}
\]
for all $x\in M$, where the second infimum is taken among the
continuous and piecewise $C^1$ paths $\gamma:[0,\tau+k]\rightarrow
M$ with $\gamma(\tau+k)=x$. For each $\tau\in[0,1]$ and each $n\in
\mathbb{N}$, $\tilde{T}_n^{a,\tau}$ is an operator from
$C(M,\mathbb{R}^1)$ to itself. Furthermore, it is not difficult to
show that for each $\tau\in[0,1]$ and each $u\in
C(M,\mathbb{R}^1)$, the uniform limit
$\lim_{n\to+\infty}\tilde{T}_n^{a,\tau}u$ exists and
$\lim_{n\to+\infty}\tilde{T}_n^{a,\tau}u=\lim_{n\to+\infty}T_n^au=\bar{u}$,
which is a backward weak KAM solution of (\ref{1-2}), see Remark
\ref{re4-1}. It means that the parameter $\tau$ does not effect
the convergence of $\{\tilde{T}_n^{a,\tau}u\}_{n\in\mathbb{N}}$.
Therefore, without any loss of generality, we take $\tau=0$ and
thus consider the operator $\tilde{T}_n^{a,0}=\tilde{T}_n^a$. In
order to compare the new family of L-O operators to the {\em full}
L-O semigroup $\{T^a_t\}_{t\geq 0}$, it is convenient to define
{\em the new L-O operator associated with $L_a$} as follows.

\begin{definition}\label{def2}
For each $u\in C(M,\mathbb{R}^1)$ and each $t\geq0$, let
\[
\tilde{T}^a_tu(x)=\inf_{t\leq \sigma\leq
2t}\inf_{\gamma}\Big\{u(\gamma(0))+\int_0^\sigma
L_a(\gamma(s),\dot{\gamma}(s))ds\Big\}
\]
for all $x\in M$, where the second infimum is taken among the
continuous and piecewise $C^1$ paths $\gamma:[0,\sigma]\rightarrow
M$ with $\gamma(\sigma)=x$.
\end{definition}

It is easy to check that $\{\tilde{T}^a_t\}_{t\geq
0}:C(M,\mathbb{R}^1)\to C(M,\mathbb{R}^1)$ is a one-parameter
semigroup of operators. We call it {\em the new L-O semigroup
associated with $L_a$}. We show that $u\in C(M,\mathbb{R}^1)$ is a
fixed point of $\{\tilde{T}^a_t\}_{t\geq 0}$ if and only if it is
a fixed point of $\{T^a_t\}_{t\geq 0}$, and that for each $u\in
C(M,\mathbb{R}^1)$, the uniform limit
$\lim_{t\to+\infty}\tilde{T}^a_tu=\lim_{t\to+\infty}T^a_tu=\bar{u}$.
For more properties of $\tilde{T}^a_t$ we refer to Section 4.

How fast does the L-O semigroup converge? It is an interesting
question which is well worth discussing. We believe that there is
a deep relation between dynamical properties of Mather sets (or
Aubry sets) and the rate of convergence of the L-O semigroup. To
the best of our knowledge there are now two relative results: In
\cite{Itu}, Iturriaga and S$\mathrm{\acute{a}}$nchez-Morgado prove
that if the Aubry set consists in a finite number of hyperbolic
periodic orbits or hyperbolic fixed points, the L-O semigroup
converges exponentially. Recently, in \cite{Wan} the authors deal
with the rate of convergence problem when the Mather set consists
of degenerate fixed points. More precisely, consider the standard
Lagrangian in classical mechanics
$L^0_a(x,v)=\frac{1}{2}v^2+U(x)$, $x\in \mathbb{S}^1$, $v\in
\mathbb{R}^1$, where $U$ is a real analytic function on
$\mathbb{S}^1$ and has a unique global minimum point $x_0$.
Without loss of generality, one may assume $x_0=0$, $U(0)=0$. Then
$c(L^0_a)=0$ and $\tilde{\mathcal{M}}_0=\{(0,0)\}$, where
$\tilde{\mathcal{M}}_0$ is the Mather set with cohomology class 0
\cite{Mat91}. An upper bound estimate of the rate of convergence
of the L-O semigroup is provided in \cite{Wan} under the
assumption that $\{(0,0)\}$ is a degenerate fixed point: for every
$u\in C(\mathbb{S}^1,\mathbb{R}^1)$, there exists a constant $C>0$
such that

\[
\|T^a_tu-\bar{u}\|_\infty\leq\frac{C}{\sqrt[k-1]{t}}, \quad
\forall t>0,
\]
where $k\in \mathbb{N}$, $k\geq 2$ depends only on the degree of
degeneracy of the minimum point of the potential function $U$.

Naturally, we also care the problem of the rate of convergence of
the new L-O semigroup. We compare the rate of convergence of the
new L-O semigroup to the rate for the L-O semigroup as follows.
First, we show that for each $u\in C(M,\mathbb{R}^1)$,
$\|\tilde{T}^a_tu-\bar{u}\|_\infty\leq\|T^a_tu-\bar{u}\|_\infty$,
$\forall t\geq 0$. It means that the new L-O semigroup converges
faster than the L-O semigroup.

Then, in particular, we consider a class of $C^2$ positive
definite and superlinear Lagrangians on $\mathbb{T}^n$

\begin{align}\label{1-6}
L^1_a(x,v)=\frac{1}{2}\langle
A(x)(v-\omega),(v-\omega)\rangle+f(x,v-\omega), \quad x\in
\mathbb{T}^n,\ v\in\mathbb{R}^n,
\end{align}
where $A(x)$ is an $n\times n$ matrix, $\omega\in\mathbb{S}^{n-1}$
is a given vector, and $f(x,v-\omega)=O(\|v-\omega\|^3)$ as
$v-\omega\rightarrow 0$. It is clear that $c(L^1_a)=0$ and
$\tilde{\mathcal{M}}_0=\tilde{\mathcal{A}}_0=\tilde{\mathcal{N}}_0=\cup_{x\in\mathbb{T}^n}(x,\omega)$,
which is a quasi-periodic invariant torus with frequency vector
$\omega$ of the Euler-Lagrange flow associated to $L^1_a$, where
$\tilde{\mathcal{A}}_0$ and $\tilde{\mathcal{N}}_0$ are the Aubry
set and the Ma$\mathrm{\tilde{n}}\mathrm{\acute{e}}$ set with
cohomology class 0 \cite{Mat93}, respectively. For the Lagrangian
system (\ref{1-6}), we obtain the following two results on the
rates of convergence of the L-O semigroup and the new L-O
semigroup, respectively.

\begin{theorem}\label{th2}
For each $u\in C(\mathbb{T}^n,\mathbb{R}^1)$, there is a constant
$K>0$ such that

\[
\|T^a_tu-\bar{u}\|_\infty\leq\frac{K}{t}, \quad \forall t>0,
\]
where $K$ depends only on $n$ and $u$.
\end{theorem}

We recall the notations for Diophantine vectors: for $\rho>n-1$
and $\alpha>0$, let

\[
\mathcal{D}(\rho,\alpha)=\Big\{\beta\in \mathbb{S}^{n-1}\ |\
|\langle\beta,k\rangle|\geq\frac{\alpha}{|k|^\rho},\ \forall
k\in\mathbb{Z}^n\backslash\{0\}\Big\},
\]
where $|k|=\sum_{i=1}^n|k_i|$.

\begin{theorem}\label{th3}
Given any frequency vector $\omega\in\mathcal{D}(\rho,\alpha)$,
for each $u\in C(\mathbb{T}^n,\mathbb{R}^1)$, there is a constant
$\tilde{K}>0$ such that

\[
\|\tilde{T}^a_tu-\bar{u}\|_\infty\leq
\tilde{K}t^{-(1+\frac{4}{2\rho+n})}, \quad \forall t>0,
\]
where $\tilde{K}$ depends only on $n$, $\rho$, $\alpha$ and $u$.
\end{theorem}

Finally, we construct an example (Example \ref{ex1}) to show that
the result of Theorem \ref{th2} is sharp in the sense of order.
Therefore, in view of Theorems \ref{th2}, \ref{th3} and Example
\ref{ex1}, we conclude that the new L-O semigroup converges faster
than the L-O semigroup in the sense of order when the Aubry set
$\tilde{\mathcal{A}}_0$ of the Lagrangian system (\ref{1-6}) is a
quasi-periodic invariant torus with Diophantine frequency vector
$\omega\in\mathcal{D}(\rho,\alpha)$.

We hope that the new L-O operator introduced in the present paper
will contribute to the development of the Mather theory and the
weak KAM theory. At the end of this section, we refer the reader
to some good introductory books (lecture notes), survey articles
and most recent research articles on the Mather theory and the
weak KAM theory:
\cite{Mat-b,Fat-b,Con-b,Sor-b,Man92,Man96,Eva04,Eva08,Kal,Arn11,Arn,Ber08,Ber082,Eva09,Gom08,Gom10}.

The rest of the paper is organized as follows. In Section 2 we
introduce the basic language and notation used in the sequel. In
Section 3 we first study the basic properties of the new L-O
operator associated with $L$ and then prove Theorem \ref{th1}. The
last part of the section is devoted to the discussion of the
relation among uniform limits $\lim_{n\to+\infty}U^u_n$, backward
weak KAM solutions and viscosity solutions of (\ref{1-5}). In
Section 4 we first study the basic properties of the new L-O
semigroup associated with $L_a$ and then give the proofs of
Theorems \ref{th2} and \ref{th3}. At last, we construct the
example mentioned above (Example \ref{ex1}).

\section{Notation and terminology}
Consider the flat $n$-torus $\mathbb{T}^n$, whose universal cover
is the Euclidean space $\mathbb{R}^n$. We view the torus as a
fundamental domain in $\mathbb{R}^n$

\[
\overline{A}=\underbrace{[0,1]\times\dots\times[0,1]}_{n\
\mathrm{times}}
\]
with opposite faces identified. The unique coordinates
$x=(x_1,\dots,x_n)$ of a point in $\mathbb{T}^n$ will belong to
the half-open cube

\[
A=\underbrace{[0,1)\times\dots\times[0,1)}_{n\ \mathrm{times}}.
\]
In these coordinates the standard universal covering projection
$\pi:\mathbb{R}^n\rightarrow\mathbb{T}^n$ takes the form

\[
\pi(\tilde{x})=([\tilde{x}_1],\dots,[\tilde{x}_n]),
\]
where $[\tilde{x}_i]=\tilde{x}_i$ mod 1, denotes the fractional
part of $\tilde{x}_i$
($\tilde{x}_i=\{\tilde{x}_i\}+[\tilde{x}_i]$, where
$\{\tilde{x}_i\}$ is the greatest integer not greater than
$\tilde{x}_i$). We can now define operations on $\mathbb{T}^n$
using the covering projection: each operation is simply the
projection of the usual operation with coordinates in
$\mathbb{R}^n$. Thus the flat metric $d_{\mathbb{T}^n}$ may be
defined for any pair of points $x$, $y\in\mathbb{T}^n$ by
$d_{\mathbb{T}^n}(x,y)=\|x-y\|$, where $\|\cdot\|$ is the usual
Euclidean norm on $\mathbb{R}^n$. And the distance between points
on the torus is at most $\frac{\sqrt{n}}{2}$. For
$x\in\mathbb{T}^n$ and $R>0$, $B_R(x)=\{y\in\mathbb{T}^n|\
d_{\mathbb{T}^n}(x,y)<R\}$ denotes the open ball of the radius $R$
centered on $x$ in $\mathbb{T}^n$.

We choose, once and for all, a $C^\infty$ Riemannian metric on
$M$. It is classical that there is a canonical way to associate to
it a Riemannian metric on $TM$. We use the same symbol ``$d$" to
denote the distance function defined by the Riemannian metric on
$M$ and the distance function defined by the Riemannian metric on
$TM$. Denote by $\|\cdot\|_x$ the norm induced by the Riemannian
metric on the fiber $T_xM$ for $x\in M$, and by $\langle
\cdot,\cdot\rangle_x$ the canonical pairing between $T_xM$ and
$T_x^*M$. In particular, for $M=\mathbb{T}^n$, we denote $\langle
\cdot,\cdot\rangle_x$ by $\langle \cdot,\cdot\rangle$ for brevity.
We use the same notation $\langle \cdot,\cdot\rangle$ for the
standard inner product on $\mathbb{R}^n$. However, this should not
create any ambiguity.

We equip $C(M,\mathbb{R}^1)$ and
$C(M\times\mathbb{S}^1,\mathbb{R}^1)$ with the usual uniform
topology (the compact-open topology, or the $C^0$-topology)
defined by the supremum norm $\|\cdot\|_\infty$. We use $u\equiv
const.$ to denote a constant function whose values do not vary.

\section{The new L-O operator: time-periodic case}
In this section we first discuss some basic properties of the new
L-O operator associated with $L$, i.e., $\{\tilde{T}_n^\tau\}$,
and then study the uniform convergence of $U_n^u$, $\forall u\in
C(M,\mathbb{R}^1)$, as $n\to+\infty$. At last, we discuss the
relation among uniform limits $\lim_{n\to+\infty}U^u_n$, backward
weak KAM solutions and viscosity solutions of (\ref{1-5}).

\subsection{Basic properties of the new L-O operator}
Recall the definition of the new L-O operator associated with $L$.
For each $\tau\in[0,1]$, each $n\in \mathbb{N}$ and each $u\in
C(M,\mathbb{R}^1)$,

\[
\tilde{T}_n^\tau u(x)=\inf_{k\in\mathbb{N} \atop n\leq k\leq
2n}\inf_{\gamma}\Big\{u(\gamma(0))+\int_{0}^{\tau+k}
L(\gamma(s),\dot{\gamma}(s),s)ds\Big\}
\]
for all $x\in M$, where the second infimum is taken among the
continuous and piecewise $C^1$ paths $\gamma:[0,\tau+k]\rightarrow
M$ with $\gamma(\tau+k)=x$.

First of all, we show that for each $\tau\in[0,1]$ and each
$n\in\mathbb{N}$, $\tilde{T}_n^\tau$ is an operator from
$C(M,\mathbb{R}^1)$ to itself. For this, noticing that
$\tilde{T}_n^\tau u(x)\in\mathbb{R}^1$ for all $x\in M$, we only
need to prove the following result.

\begin{proposition}\label{pr3-1}
For each $\tau\in[0,1]$, each $n\in \mathbb{N}$ and each $u\in
C(M,\mathbb{R}^1)$, $\tilde{T}_n^\tau u$ is a continuous function
on $M$.
\end{proposition}

\begin{proof}
Following Mather (\cite{Mat93}, also see \cite{Ber}), it is
convenient to introduce, for $t'\geq t$ and $x$, $y\in M$, the
following quantity:

\[
F_{t,t'}(x,y)=\inf_\gamma\int_t^{t'}L(\gamma(s),\dot{\gamma}(s),s)ds,
\]
where the infimum is taken over the continuous and piecewise $C^1$
paths $\gamma:[t,t']\to M$ such that $\gamma(t)=x$ and
$\gamma(t')=y$.

By the definition of $\tilde{T}_n^\tau$, for each $u\in
C(M,\mathbb{R}^1)$ and each $x\in M$, we have

\begin{align*}
\tilde{T}_n^\tau u(x)=\inf_{k\in\mathbb{N} \atop n\leq k\leq
                       2n}\inf_{y\in M}\big(u(y)+F_{0,\tau+k}(y,x)\big).
\end{align*}
Since the function $(y,x)\mapsto F_{0,\tau+k}(y,x)$ is continuous
for each $n\leq k\leq 2n$, $k\in\mathbb{N}$ (see \cite{Ber}), then
from the compactness of $M$ the function $x\mapsto \inf_{y\in
M}\big(u(y)+F_{0,\tau+k}(y,x)\big)$ is also continuous for each
$n\leq k\leq 2n$, $k\in\mathbb{N}$. Therefore, the function
$x\mapsto\tilde{T}_n^\tau u(x)$ is continuous on $M$.
\end{proof}

\begin{proposition}\label{pr3-2}
For given $\tau\in[0,1]$, $n\in\mathbb{N}$, $u\in
C(M,\mathbb{R}^1)$ and $x\in M$, there exist $n\leq k_0\leq2n$,
$k_0\in\mathbb{N}$ and a minimizing extremal curve
$\gamma:[0,\tau+k_0]\to M$ such that $\gamma(\tau+k_0)=x$ and

\[
\tilde{T}_n^\tau u(x)=u(\gamma(0))+\int_{0}^{\tau+k_0}
L(\gamma(s),\dot{\gamma}(s),s)ds.
\]
\end{proposition}

\begin{proof}
Recall that

\begin{align*}
\tilde{T}_n^\tau u(x)=\inf_{k\in\mathbb{N} \atop n\leq k\leq
                       2n}\inf_{y\in M}\Big(u(y)+F_{0,\tau+k}(y,x)\Big).
\end{align*}

For each $k$, the function $y\mapsto u(y)+F_{0,\tau+k}(y,x)$ is
continuous on $M$. Thus, from the compactness of $M$ there exist
$y^k\in M$ such that

\begin{align*}
\tilde{T}_n^\tau u(x)=\inf_{k\in\mathbb{N} \atop n\leq k\leq
                       2n}\Big(u(y^k)+F_{0,\tau+k}(y^k,x)\Big).
\end{align*}
Then it is clear that there is $n\leq k_0\leq2n$,
$k_0\in\mathbb{N}$ such that

\[
\tilde{T}_n^\tau u(x)=u(y^{k_0})+F_{0,\tau+k_0}(y^{k_0},x).
\]
It follows from Tonelli's theorem (see, for example, \cite{Mat91})
that there exists a minimizing extremal curve
$\gamma:[0,\tau+k_0]\to M$ such that $\gamma(0)=y^{k_0}$,
$\gamma(\tau+k_0)=x$ and

\[
F_{0,\tau+k_0}(y^{k_0},x)=\int_{0}^{\tau+k_0}
L(\gamma(s),\dot{\gamma}(s),s)ds.
\]
Hence,

\[
\tilde{T}_n^\tau u(x)=u(\gamma(0))+\int_{0}^{\tau+k_0}
L(\gamma(s),\dot{\gamma}(s),s)ds.
\]
\end{proof}

\begin{proposition}\label{pr3-3}
\noindent\begin{itemize}
    \item [(1)] For $u$, $v\in C(M,\mathbb{R}^1)$, if $u\leq
              v$, then $\tilde{T}_n^\tau u\leq\tilde{T}_n^\tau v$, $\forall
              \tau\in[0,1]$, $\forall n\in\mathbb{N}$.
    \item [(2)] If $c$ is a constant and $u\in C(M,\mathbb{R}^1)$,
              then $\tilde{T}_n^\tau(u+c)=\tilde{T}_n^\tau u+c$, $\forall
              \tau\in[0,1]$, $\forall n\in\mathbb{N}$.
    \item [(3)] For each $u$, $v\in C(M,\mathbb{R}^1)$,
              $\|\tilde{T}_n^\tau u-\tilde{T}_n^\tau
              v\|_\infty\leq\|u-v\|_\infty$, $\forall
              \tau\in[0,1]$, $\forall n\in\mathbb{N}$.
\end{itemize}
\end{proposition}

\begin{proof}
For each $\tau\in[0,1]$, each $n\in\mathbb{N}$ and each $x\in M$,

\begin{align*}
\tilde{T}_n^\tau u(x) & =\inf_{k\in\mathbb{N} \atop n\leq k\leq
                       2n}\inf_{y\in
                       M}\big(u(y)+F_{0,\tau+k}(y,x)\big)\\
                      & \leq \inf_{k\in\mathbb{N} \atop n\leq k\leq
                       2n}\inf_{y\in
                       M}\big(v(y)+F_{0,\tau+k}(y,x)\big)\\
                      &= \tilde{T}_n^\tau v(x),
\end{align*}
which proves (1). (2) results from the definition of
$\tilde{T}_n^\tau$ directly. To prove (3), we notice that for each
$x\in M$,

\[
-\|u-v\|_\infty+v(x)\leq u(x)\leq\|u-v\|_\infty+v(x).
\]
From (1) and (2), for each $x\in M$ we have

\[
\tilde{T}_n^\tau v(x)-\|u-v\|_\infty\leq \tilde{T}_n^\tau u(x)\leq
\tilde{T}_n^\tau v(x)+\|u-v\|_\infty, \quad \forall \tau\in[0,1],
\ \forall n\in\mathbb{N}.
\]
Hence, $\|\tilde{T}_n^\tau u-\tilde{T}_n^\tau
v\|_\infty\leq\|u-v\|_\infty$, $\forall \tau\in[0,1]$, $\forall
n\in\mathbb{N}$.
\end{proof}

\subsection{Uniform convergence of $U^u_n$}
Here we deal with the uniform convergence of $U^u_n$, $\forall
u\in C(M,\mathbb{R}^1)$, as $n\to+\infty$. We show that for each
$u\in C(M,\mathbb{R}^1)$ the uniform limit
$\bar{u}=\lim_{n\to+\infty}U^u_n$ exists and

\[
\bar{u}(x,\tau)=\inf_{y\in M}\big(u(y)+h_{0,\tau}(y,x)\big)
\]
for all $(x,\tau)\in M\times\mathbb{S}^1$. This is an immediate
consequence of Proposition \ref{pr3-4} below.

Following Ma$\mathrm{\tilde{n}}\mathrm{\acute{e}}$ \cite{Man97}
and Mather \cite{Mat93}, define the action potential and the
extended Peierls barrier as follows.

\vskip0.1cm

\noindent {\em Action Potential}: for each
$(\tau,\tau')\in\mathbb{S}^1\times\mathbb{S}^1$, let

\[
\Phi_{\tau,\tau'}(x,x')=\inf F_{t,t'}(x,x')
\]
for all $(x,x')\in M\times M$, where the infimum is taken on the
set of $(t,t')\in\mathbb{R}^2$ such that $\tau=[t]$, $\tau'=[t']$
and $t'\geq t+1$.

\vskip0.1cm

\noindent {\em Extended Peierls Barrier}: for each
$(\tau,\tau')\in\mathbb{S}^1\times\mathbb{S}^1$, let

\begin{align}\label{3-1}
 h_{\tau,\tau'}(x,x')=\liminf_{t'-t\to+\infty}F_{t,t'}(x,x')
\end{align}
for all $(x,x')\in M\times M$, where the liminf is restricted to
the set of $(t,t')\in\mathbb{R}^2$ such that $\tau=[t]$,
$\tau'=[t']$.

\vskip0.1cm

From the above definitions, it is not hard to see that

\begin{align}\label{3-2}
\Phi_{\tau,\tau'}(x,x')\leq h_{\tau,\tau'}(x,x'), \quad \forall
(x,\tau),\ (x',\tau')\in M\times\mathbb{S}^1
\end{align}
and

\begin{align}\label{3-3}
h_{\tau,t}(x,y)\leq h_{\tau,s}(x,z)+\Phi_{s,t}(z,y), \quad \forall
(x,\tau),\ (y,t),\ (z,s)\in M\times\mathbb{S}^1.
\end{align}
It can be shown that the extended Peierls barrier $h_{\tau,\tau'}$
is Lipschitz and that, the liminf in (\ref{3-1}) can not always be
replaced with a limit, which leads to the non-convergence of the
L-O semigroup associated with $L$ \cite{Fat5}. See \cite{Sor-b}
for more details about the action potential and the extended
Peierls barrier. Before stating Proposition \ref{pr3-4}, we
introduce the following lemma.

\begin{lemma}[A Priori Compactness]\label{le3-1}
If $t>0$ is fixed, there exists a compact subset
$\mathcal{C}_t\subset TM\times\mathbb{S}^1$ such that for each
minimizing extremal curve $\gamma:[a,b]\to M$ with $b-a\geq t$, we
have

\[
(\gamma(s),\dot{\gamma}(s),[s])\in \mathcal{C}_t, \quad \forall
s\in[a,b].
\]
\end{lemma}

The lemma may be proved by small modifications of the proof found
in \cite[Corollary 4.3.2]{Fat-b}.

\begin{proposition}\label{pr3-4}
\[
\lim_{n\to+\infty}\inf_{k\in\mathbb{N} \atop n\leq k\leq
2n}F_{\tau,\tau'+k}(x,x')=h_{\tau,\tau'}(x,x')
\]
uniformly on  $(\tau,\tau',x,x')\in
\mathbb{S}^1\times\mathbb{S}^1\times M\times M$.
\end{proposition}

\begin{proof}
Throughout this proof we use $C$ to denote a generic positive
constant not necessarily the same in any two places. Since the
proof is rather long, it is convenient to divide it into two
steps.

Step 1. In the first step, we show that

\begin{align}\label{3-4}
\lim_{n\to+\infty}\inf_{k\in\mathbb{N} \atop n\leq k\leq
2n}F_{\tau,\tau'+k}(x,x')=h_{\tau,\tau'}(x,x'), \quad \forall
(\tau,\tau',x,x')\in \mathbb{S}^1\times\mathbb{S}^1\times M\times
M.
\end{align}
For each $\tau,\tau'\in\mathbb{S}^1$ and each $x$, $x'\in M$, by
the definition of $h_{\tau,\tau'}$, we have
$\liminf_{k\to+\infty}F_{\tau,\tau'+k}(x,x')=h_{\tau,\tau'}(x,x')$.
Then there exist $\{k_i\}_{i=1}^{+\infty}$ such that
$k_i\to+\infty$ and $F_{\tau,\tau'+k_i}(x,x')\to
h_{\tau,\tau'}(x,x')$ as $i\to+\infty$. Tonelli's theorem
guarantees the existence of the minimizing extremal curves
$\gamma_{k_i}:[\tau,\tau'+k_i]\to M$ with $\gamma_{k_i}(\tau)=x$,
$\gamma_{k_i}(\tau'+k_i)=x'$ and
$A(\gamma_{k_i})=F_{\tau,\tau'+k_i}(x,x')$, where

\[
A(\gamma_{k_i})=\int_{\tau}^{\tau'+k_i}
L(\gamma_{k_i},\dot{\gamma}_{k_i},s)ds.
\]
Thus, we have $A(\gamma_{k_i})\to h_{\tau,\tau'}(x,x')$ as
$i\to+\infty$. Then for every $\varepsilon>0$, there exists
$I\in\mathbb{N}$ such that

\[
|A(\gamma_{k_i})-h_{\tau,\tau'}(x,x')|<\varepsilon
\]
if $i\geq I$, $i\in \mathbb{N}$. And it is clear that for each
$k_i$,
$(\gamma_{k_i}(s),\dot{\gamma}_{k_i}(s),[s]):[\tau,\tau'+k_i]\to
TM\times\mathbb{S}^1$ is a trajectory of the Euler-Lagrange flow.

To prove (\ref{3-4}), it suffices to show that for
$n\in\mathbb{N}$ large enough, we can find a curve
$\tilde{\gamma}:[\tau,\tau'+k_0]\to M$ with
$\tilde{\gamma}(\tau)=x$, $\tilde{\gamma}(\tau'+k_0)=x'$, where
$n\leq k_0\leq2n$, $k_0\in\mathbb{N}$, such that

\[
|A(\tilde{\gamma})-A(\gamma_{k_I})|\leq C\varepsilon
\]
for some constant $C>0$. In fact, if such a curve exists, then

\[
\inf_{k\in\mathbb{N} \atop n\leq
k}F_{\tau,\tau'+k}(x,x')\leq\inf_{k\in\mathbb{N} \atop n\leq k\leq
2n}F_{\tau,\tau'+k}(x,x')\leq A(\tilde{\gamma})\leq
A(\gamma_{k_I})+C\varepsilon\leq
h_{\tau,\tau'}(x,x')+C\varepsilon.
\]
By letting $n\to+\infty$, from the arbitrariness of
$\varepsilon>0$, we have

\begin{align*}
h_{\tau,\tau'}(x,x')& = \liminf_{k\to+\infty}F_{\tau,\tau'+k}(x,x')\\
                    & = \lim_{n\to+\infty}\inf_{k\in\mathbb{N} \atop n\leq
                      k}F_{\tau,\tau'+k}(x,x')\\
                    & \leq \lim_{n\to+\infty}\inf_{k\in\mathbb{N}
                    \atop n\leq k\leq 2n}F_{\tau,\tau'+k}(x,x')\\
                    & \leq h_{\tau,\tau'}(x,x'),
\end{align*}
which implies that

\[
\lim_{n\to+\infty}\inf_{k\in\mathbb{N} \atop n\leq k\leq
2n}F_{\tau,\tau'+k}(x,x')=h_{\tau,\tau'}(x,x').
\]

Our task is now to construct the curve mentioned above. Note that
for the above $\varepsilon>0$, there exists $I'\in\mathbb{N}$ such
that there exists

\[
(z_{k_i},v_{z_{k_i}},t_{z_{k_i}})\in
O_i:=\{(\gamma_{k_i}(s),\dot{\gamma}_{k_i}(s),[s])\ |\ \tau\leq
s\leq\tau'+k_i\}\subset TM\times\mathbb{S}^1
\]
such that

\[
d((z_{k_i},v_{z_{k_i}},t_{z_{k_i}}),\tilde{\mathcal{M}}_0)<\varepsilon,
\]
if $i\geq I'$, $i\in\mathbb{N}$, where $\tilde{\mathcal{M}}_0$ is
the Mather set of cohomology class 0. As usual, distance is
measured with respect to smooth Riemannian metrics. Since
$\tilde{\mathcal{M}}_0$ is compact and by the a priori compactness
given by Lemma \ref{le3-1}, $O_i$ is contained in the compact
subset $\mathcal{C}_{k_{I'}-1}$ of $TM\times\mathbb{S}^1$ for each
$i\geq I'$, then it doesn't matter which Riemannian metrics we
choose to measure distance.

Let $I=\max\{I,I'\}$. Then
$|A(\gamma_{k_I})-h_{\tau,\tau'}(x,x')|<\varepsilon$ and there
exists $(z_0,v_{z_0},t_{z_0})\in
O_I=\{(\gamma_{k_I}(s),\dot{\gamma}_{k_I}(s),[s])\ |\ \tau\leq
s\leq\tau'+k_I\}$ such that

\begin{align}\label{3-5}
d((z_0,v_{z_0},t_{z_0}),\tilde{\mathcal{M}}_0)<\varepsilon.
\end{align}

In view of (\ref{3-5}), there exists an ergodic minimal measure
$\mu_e$ on $TM\times\mathbb{S}^1$ \cite{Mat91} such that
$\mu_e(\mathrm{supp}\mu_e\cap
B_{2\varepsilon}(z_0,v_{z_0},t_{z_0}))=\Delta>0$, where
$B_{2\varepsilon}(z_0,v_{z_0},t_{z_0})$ denotes the open ball of
radius $2\varepsilon$ centered on $(z_0,v_{z_0},t_{z_0})$ in
$TM\times\mathbb{S}^1$. Set
$A_{2\varepsilon}=\mathrm{supp}\mu_e\cap
B_{2\varepsilon}(z_0,v_{z_0},t_{z_0})$. Since $\mu_e$ is an
ergodic measure, then

\[
\mu_e(\bigcup_{t=1}^{+\infty}\phi^L_{-t}(A_{2\varepsilon}))=1.
\]
Thus, for any $0<\Delta'<\Delta$, there exists $T>0$ such that

\[
\mu_e(\bigcup_{t=1}^{T'}\phi^L_{-t}(A_{2\varepsilon}))\geq
1-\Delta',
\]
if $T'\geq T$. From this, we may deduce that for each
$n\in\mathbb{N}$,

\begin{align}\label{3-6}
\Big(\bigcup_{t=1}^T\phi^L_{-t}(A_{2\varepsilon})\Big)\cap\phi^L_n(A_{2\varepsilon})
\neq\emptyset.
\end{align}
For, otherwise, there would be $n_0\in\mathbb{N}$ such that

\begin{align*}
\mu_e\Big(\big(\bigcup_{t=1}^T\phi^L_{-t}(A_{2\varepsilon})\big)\cup\phi^L_{n_0}(A_{2\varepsilon})\Big) & =\mu_e\Big(\bigcup_{t=1}^T\phi^L_{-t}(A_{2\varepsilon})\Big)+\mu_e(\phi^L_{n_0}(A_{2\varepsilon}))\\
                                                                                                        & \geq 1-\Delta'+\Delta>1,
\end{align*}
which contradicts that $\mu_e$ is a probability measure.

For a given $n\in\mathbb{N}$ large enough with
$\max\{k_I,T+1\}\leq\{\frac{n}{2}\}$, from (\ref{3-6}) there exist
$(e_0,v_{e_0},t_{e_0})$,
$(\bar{e}_0,v_{\bar{e}_0},t_{\bar{e}_0})\in A_{2\varepsilon}$ and
$1\leq t\leq T$ such that

\begin{align}\label{3-7}
\phi^L_{-t}(e_0,v_{e_0},t_{e_0})=(e,v_e,t_e)=\phi^L_n(\bar{e}_0,v_{\bar{e}_0},t_{\bar{e}_0})
\end{align}
for some $(e,v_e,t_e)\in \tilde{\mathcal{M}}_0$. Since
$(e_0,v_{e_0},t_{e_0})\in A_{2\varepsilon}$, then

\begin{align}\label{3-8}
d((e_0,v_{e_0},t_{e_0}),(z_0,v_{z_0},t_{z_0}))<2\varepsilon.
\end{align}
Set
$(z_1,v_{z_1},t_{z_1})=\phi^L_{t_{e_0}-t_{z_0}}(z_0,v_{z_0},t_{z_0})$.
Then $t_{z_1}=t_{e_0}$ and from (\ref{3-8}) we have

\begin{align}\label{3-9}
d((e_0,v_{e_0},t_{e_0}),(z_1,v_{z_1},t_{e_0}))<C\varepsilon
\end{align}
for some constant $C>0$.  Set
$(z_2,v_{z_2},\tau)=\phi^L_{\tau-t_{e_0}}(z_1,v_{z_1},t_{e_0})$
and
$(e_1,v_{e_1},\tau)=\phi^L_{\tau-t_{e_0}}(e_0,v_{e_0},t_{e_0})$.
Then by the differentiability of the solutions of the
Euler-Lagrange equation with respect to initial values, we have

\begin{align}\label{3-10}
d((e_1,v_{e_1},\tau),(z_2,v_{z_2},\tau))<C\varepsilon
\end{align}
for some constant $C>0$.

Since $(e_0,v_{e_0},t_{e_0})$,
$(\bar{e}_0,v_{\bar{e}_0},t_{\bar{e}_0})\in A_{2\varepsilon}$,
then

\begin{align}\label{3-11}
d((e_0,v_{e_0},t_{e_0}),
(\bar{e}_0,v_{\bar{e}_0},t_{\bar{e}_0}))<4\varepsilon.
\end{align}
Set
$(\bar{e}_1,v_{\bar{e}_1},t_{e_0})=\phi^L_{t_{e_0}-t_{\bar{e}_0}}(\bar{e}_0,v_{\bar{e}_0},t_{\bar{e}_0})$.
Then from (\ref{3-11}) we have

\begin{align}\label{3-12}
d((e_0,v_{e_0},t_{e_0}),
(\bar{e}_1,v_{\bar{e}_1},t_{e_0}))<C\varepsilon
\end{align}
for some constant $C>0$. Set
$(\bar{e}_2,v_{\bar{e}_2},\tau)=\phi^L_{\tau-t_{e_0}}(\bar{e}_1,v_{\bar{e}_1},t_{e_0})$.
Recall that
$(e_1,v_{e_1},\tau)=\phi^L_{\tau-t_{e_0}}(e_0,v_{e_0},t_{e_0})$.
Then from the differentiability of the solutions of the
Euler-Lagrange equation with respect to initial values, we have

\begin{align}\label{3-13}
d((e_1,v_{e_1},\tau),(\bar{e}_2,v_{\bar{e}_2},\tau)<C\varepsilon
\end{align}
for some constant $C>0$.

Note that since $(z_0,v_{z_0},t_{z_0})\in
O_I=\{(\gamma_{k_I}(s),\dot{\gamma}_{k_I}(s),[s])\ |\ \tau\leq
s\leq\tau'+k_I\}$, where $O_I$ is an orbit of the Euler-Lagrange
flow, then $(z_2,v_{z_2},\tau)\in O_I$. And thus, there exists
$k_{I_1}$, $k_{I_2}\in\mathbb{N}$ with $k_{I_1}+k_{I_2}=k_I$ such
that

\[
(z_2,v_{z_2},\tau)=(\gamma_{k_I}(\tau+k_{I_1}),\dot{\gamma}_{k_I}(\tau+k_{I_1}),\tau).
\]

We are now in a position to construct the curve we need. We treat
the case $k_{I_1}\neq 0$, $k_{I_2}\neq 0$ and the remaining cases
can be treated similarly. Let $\alpha_1:[\tau,\tau+k_{I_1}]\to M$
with $\alpha_1(\tau)=x$ and $\alpha_1(\tau+k_{I_1})=\bar{e}_2$ be
a Tonelli minimizer such that
$A(\alpha_1)=F_{\tau,\tau+k_{I_1}}(x,\bar{e}_2)$. Since
$\gamma_{k_I}:[\tau,\tau'+k_I]\to M$ is a minimizing extremal
curve, then $\gamma_{k_I}|_{[\tau,\tau+k_{I_1}]}$ is also a
minimizing extremal curve and thus
$A(\gamma_{k_I}|_{[\tau,\tau+k_{I_1}]})=F_{\tau,\tau+k_{I_1}}(x,z_2)$.
Therefore, by the Lipschtiz property of the function
$F_{\tau,\tau+k_{I_1}}$ (see, for example, \cite{Ber}),
(\ref{3-10}) and (\ref{3-13}) we have

\begin{align}\label{3-14}
|A(\alpha_1)-A(\gamma_{k_I}|_{[\tau,\tau+k_{I_1}]})|=|F_{\tau,\tau+k_{I_1}}(x,\bar{e}_2)-F_{\tau,\tau+k_{I_1}}(x,z_2)|\leq
Dd(\bar{e}_2,z_2)\leq C\varepsilon
\end{align}
for some constant $C>0$, where $D>0$ is a Lipschitz constant of
$F_{t_1,t_2}$ which is independent of $t_1$, $t_2$ with $t_1+1\leq
t_2$.

Let
$\beta(s)=p\phi^L_{s-(\tau+k_{I_1})}(\bar{e}_2,v_{\bar{e}_2},\tau)$,
$s\in\mathbb{R}^1$, where $p:TM\times\mathbb{S}^1\to M$ denotes
the projection. Then
$(\beta(s),\dot{\beta}(s),[s])=\phi^L_{s-(\tau+k_{I_1})}(\bar{e}_2,v_{\bar{e}_2},\tau)$,
$s\in\mathbb{R}^1$, and
$(\beta(\tau+k_{I_1}),\dot{\beta}(\tau+k_{I_1}))=(\bar{e}_2,v_{\bar{e}_2})$.
Hence , from (\ref{3-7}) we have

\[
(e,v_e,t_e)=(\beta(l),\dot{\beta}(l),[l]),
\]
where $l=\tau+k_{I_1}+(t_{e_0}-\tau)+(t_{\bar{e}_0}-t_{e_0})+n$,
and

\begin{align*}
(e_1,v_{e_1},\tau) = (\beta(l'),\dot{\beta}(l'),[l']),
\end{align*}
where
$l'=l+t+(\tau-t_{e_0})=\tau+k_{I_1}+n+t+t_{\bar{e}_0}-t_{e_0}$.
Then
$[l']=[\tau+k_{I_1}+n+t+t_{\bar{e}_0}-t_{e_0}]=[\tau+t+t_{\bar{e}_0}-t_{e_0}]=\tau$,
which means that $t+t_{\bar{e}_0}-t_{e_0}\in\mathbb{Z}$. Notice
that $0\leq t+t_{\bar{e}_0}-t_{e_0}\leq
T+t_{\bar{e}_0}-t_{e_0}\leq\{\frac{n}{2}\}$. Thus,

\begin{align}\label{3-15}
n\leq k_I+n+t+t_{\bar{e}_0}-t_{e_0}\leq k_I+n+\{\frac{n}{2}\}\leq
2n.
\end{align}
Let $m=n+t+t_{\bar{e}_0}-t_{e_0}\in\mathbb{Z}$ and
$\alpha_2=\beta|_{[\tau+k_{I_1},\tau+k_{I_1}+m]}$. Then
$\alpha_2(\tau+k_{I_1})=\beta(\tau+k_{I_1})=\bar{e}_2$ and
$\alpha_2(\tau+k_{I_1}+m)=\beta(\tau+k_{I_1}+m)=e_1$. In view of
$(\bar{e}_0,v_{\bar{e}_0},t_{\bar{e}_0})\in
A_{2\varepsilon}\subset\tilde{\mathcal{M}}_0$ and the definitions
of $\beta$ and $\alpha_2$, $(\alpha_2(s),\dot{\alpha}_2(s),[s])$
is a trajectory of the Euler-Lagrange flow in
$\tilde{\mathcal{M}}_0$. According to \cite[Proposition 3]{Mat91}
and the definition of $h_{\tau,\tau}$, we have

\[
A(\alpha_2)=F_{\tau+k_{I_1},\tau+k_{I_1}+m}(\bar{e}_2,e_1)=h_{\tau,\tau}(\bar{e}_2,e_1).
\]
Hence, on account of the Lipschitz property of $h_{\tau,\tau}$ and
(\ref{3-13}),

\[
|A(\alpha_2)-h_{\tau,\tau}(e_1,e_1)|=
|h_{\tau,\tau}(\bar{e}_2,e_1)-h_{\tau,\tau}(e_1,e_1)|\leq\bar{D}d(\bar{e}_2,e_1)\leq
C\varepsilon
\]
for some constant $C>0$, where $\bar{D}$ is a Lipschitz constant
of $h_{\tau,\tau}$. Since $(e_1,\tau)\in\mathcal{M}_0$, where
$\mathcal{M}_0\subset M\times\mathbb{S}^1$ is the projected Mather
set, then $h_{\tau,\tau}(e_1,e_1)=0$, and thus

\begin{align}\label{3-16}
|A(\alpha_2)|\leq C\varepsilon.
\end{align}

Let $\alpha_3:[\tau+k_{I_1}+m,\tau'+k_I+m]\to M$ with
$\alpha_3(\tau+k_{I_1}+m)=e_1$ and $\alpha_3(\tau'+k_I+m)=x'$ be a
Tonelli minimizer such that

\[
A(\alpha_3)=F_{\tau+k_{I_1}+m,\tau'+k_I+m}(e_1,x')=F_{\tau+k_{I_1},\tau'+k_I}(e_1,x').
\]
Since $\gamma_{k_I}:[\tau,\tau'+k_I]\to M$ is a minimizing
extremal curve, then $\gamma_{k_I}|_{[\tau+k_{I_1},\tau'+k_I]}$ is
also a minimizing extremal curve and thus

\[
A(\gamma_{k_I}|_{[\tau+k_{I_1},\tau'+k_I]})=F_{\tau+k_{I_1},\tau'+k_I}(z_2,x').
\]
Therefore, from the Lipschitz property of
$F_{\tau+k_{I_1},\tau'+k_I}$ and (\ref{3-10}), we have

\begin{align}\label{3-17}
\begin{split}
|A(\alpha_3)-A(\gamma_{k_I}|_{[\tau+k_{I_1},\tau'+k_I]})|
&=|F_{\tau+k_{I_1},\tau'+k_I}(e_1,x')-F_{\tau+k_{I_1},\tau'+k_I}(z_2,x')|\\
&\leq Dd(e_1,z_2)\\
&\leq C\varepsilon
\end{split}
\end{align}
for some constant $C>0$.

Consider the curve $\tilde{\gamma}:[\tau,\tau'+k_I+m]\to M$
connecting $x$ and $x'$ defined by

\[
\tilde{\gamma}(s)= \left\{\begin{array}{ll}
          \alpha_1(s),\quad & s\in[\tau,\tau+k_{I_1}],\\[2mm]
          \alpha_2(s),\quad & s\in[\tau+k_{I_1},\tau+k_{I_1}+m],\\[2mm]
          \alpha_3(s),\quad & s\in[\tau+k_{I_1}+m,\tau'+k_I+m].
\end{array}\right.
\]
By (\ref{3-15}), $n\leq k_0:=k_I+m\leq 2n$. From (\ref{3-14}),
(\ref{3-16}) and (\ref{3-17}), we have

\[
|A(\tilde{\gamma})-A(\gamma_{k_I})|\leq C\varepsilon
\]
for some constant $C>0$. It is clear that $\tilde{\gamma}$ is just
the curve we need, and we have proved (\ref{3-4}).

Step 2. For each $n\in\mathbb{N}$ and each
$(\tau,\tau',x,x')\in[0,1]\times[0,1]\times M\times M$, let

\[
\mathcal{F}_n(\tau,\tau',x,x')=\inf_{k\in\mathbb{N} \atop n\leq
k\leq 2n}F_{\tau,\tau'+k}(x,x').
\]
Then, to complete the proof of Proposition \ref{pr3-4}, it
suffices to show that $\{\mathcal{F}_n\}_{n=2}^{+\infty}$ are
equicontinuous. Notice that $(\tau,\tau',x,x')\mapsto
F_{\tau,\tau'+k}(x,x')$ is a Lipschitz function on
$[0,1]\times[0,1]\times M\times M$ for every $k\geq 2$,
$k\in\mathbb{N}$, and that the Lipschitz constant $\tilde{D}$ is
independent of $k$, see \cite[3.3 LEMMA]{Ber}. Hence, for each
$n\geq 2$, $n\in\mathbb{N}$ the function $(\tau,\tau',x,x')\mapsto
\mathcal{F}_n(\tau,\tau',x,x')$ is  also Lipschitz with the same
Lipschitz constant $\tilde{D}$, and thus
$\{\mathcal{F}_n\}_{n=2}^{+\infty}$ are equicontinuous. The proof
is now complete.

\end{proof}

Recall that for each $n\in\mathbb{N}$ and each $u\in
C(M,\mathbb{R}^1)$,

\[
U^u_n(x,\tau)=\tilde{T}_n^\tau u(x)=\inf_{k\in\mathbb{N} \atop
n\leq k\leq 2n}\inf_{y\in
M}\big(u(y)+F_{0,\tau+k}(y,x)\big)=\inf_{y\in
M}\big(u(y)+\mathcal{F}_n(0,\tau,y,x)\big)
\]
for all $(x,\tau)\in M\times[0,1]$. Since

\begin{align*}
\big|U^u_n(x,\tau)-\inf_{y\in
M}\big(u(y)+h_{0,\tau}(y,x)\big)\big|&= \big|\inf_{y\in
M}\big(u(y)+\mathcal{F}_n(0,\tau,y,x)\big)-\inf_{y\in
M}\big(u(y)+h_{0,\tau}(y,x)\big)\big|\\
& \leq \sup_{y\in M}|\mathcal{F}_n(0,\tau,y,x)-h_{0,\tau}(y,x)|,
\end{align*}
then from Proposition \ref{pr3-4}, we conclude that the uniform
limit $\bar{u}=\lim_{n\to+\infty}U^u_n$ exists, and

\begin{align}\label{3-18}
\bar{u}(x,\tau)=\inf_{y\in M}\big(u(y)+h_{0,\tau}(y,x)\big)
\end{align}
for all $(x,\tau)\in M\times\mathbb{S}^1$, thus proving the first
assertion of Theorem \ref{th1}.

\subsection{$\lim_{n\to+\infty}U_n^u$, backward weak KAM solutions and viscosity solutions}
Here we discuss the relation among uniform limits
$\lim_{n\to+\infty}U^u_n$, backward weak KAM solutions and
viscosity solutions of (\ref{1-5}). Following Fathi \cite{Fat1},
as done by Contreras et al. in \cite{Con}, we give the definition
of the backward weak KAM solution as follows.

\begin{definition}\label{def3}
A backward weak KAM solution of the Hamilton-Jacobi equation
(\ref{1-5}) is a function $u:M\times\mathbb{S}^1\to\mathbb{R}^1$
such that
\begin{itemize}
    \item [(1)] u is dominated by $L$, i.e.,
               \[
               u(x,\tau)-u(y,s)\leq\Phi_{s,\tau}(y,x), \quad
               \forall (x,\tau),\ (y,s)\in M\times\mathbb{S}^1.
               \]
               We use the notation $u\prec L$.
    \item [(2)] For every $(x,\tau)\in M\times\mathbb{S}^1$ there
               exists a curve $\gamma:(-\infty,\tilde{\tau}]\to M$ with
               $\gamma(\tilde{\tau})=x$ and $[\tilde{\tau}]=\tau$ such that
               \[
               u(x,\tau)-u(\gamma(t),[t])=\int_t^{\tilde{\tau}}L(\gamma(s),\dot{\gamma}(s),s)ds,\quad
               \forall t\in(-\infty,\tilde{\tau}].
               \]
\end{itemize}
\end{definition}
\noindent We denote by $\mathcal{S}_-$ the set of backward weak
KAM solutions. Let us recall two known results \cite{Con} on
backward weak KAM solutions, which will be used later in the
paper.

\begin{lemma}\label{le3-2}
Given a fixed $(y,s)\in M\times\mathbb{S}^1$, the function

\[
(x,\tau)\mapsto h_{s,\tau}(y,x),\quad (x,\tau)\in
M\times\mathbb{S}^1
\]
is a  backward weak KAM solution.
\end{lemma}

\begin{lemma}\label{le3-3}
If $\mathcal{U}\subset \mathcal{S}_-$, let
$\underline{u}(x,\tau):=\inf_{u\in\mathcal{U}}u(x,\tau)$ then
either $\underline{u}\equiv-\infty$ or
$\underline{u}\in\mathcal{S}_-$.
\end{lemma}

We define the projected Aubry set $\mathcal{A}_0$ as follows:

\[
\mathcal{A}_0:=\{(x,\tau)\in M\times\mathbb{S}^1\ |\
h_{\tau,\tau}(x,x)=0\}.
\]
Note that $\mathcal{A}_0=\Pi\tilde{\mathcal{A}_0}$, where
$\Pi:TM\times\mathbb{S}^1\to M\times\mathbb{S}^1$ denotes the
projection and $\tilde{\mathcal{A}_0}$ denotes the Aubry set in
$TM\times\mathbb{S}^1$, i.e., the union of global static orbits.
See for instance \cite{Ber} for the definition of static orbits
and more details on $\tilde{\mathcal{A}_0}$.

From the definition of $\mathcal{A}_0$, (\ref{3-2}) and
(\ref{3-3}), it is straightforward to show that if
$(x,\tau)\in\mathcal{A}_0$, then

\begin{align}\label{3-19}
h_{\tau,s}(x,y)=\Phi_{\tau,s}(x,y)
\end{align}
for all $(y,s)\in M\times\mathbb{S}^1$. Define an equivalence
relation on $\mathcal{A}_0$ by saying that $(x,\tau)$ and $(y,s)$
are equivalent if and only if

\begin{align}\label{3-20}
\Phi_{\tau,s}(x,y)+\Phi_{s,\tau}(y,x)=0.
\end{align}
By (\ref{3-19}), it is simple to see that (\ref{3-20}) is
equivalent to

\[
h_{\tau,s}(x,y)+h_{s,\tau}(y,x)=0.
\]
The equivalent classes of this relation are called static classes.
Let $\mathrm A$ be the set of static classes. For each static
class $\Gamma\in \mathrm A$ choose a point $(x,0)\in\Gamma$ and
let $\mathbb{A}_0$ be the set of such points.

Contreras et al. \cite{Con} characterize  backward weak KAM
solutions of the Hamilton-Jacobi equation (\ref{1-5}) in terms of
their values at each static class and the extended Peierls
barrier. See \cite{Con01} for similar results in the
time-independent case.

\begin{theorem}[Contreras et al. \cite{Con}]\label{th4}
The map $\{f:\mathbb{A}_0\to\mathbb{R}^1\ |\ f\prec
L\}\to\mathcal{S}_-$

\[
f\mapsto
u_f(x,\tau)=\min_{(p,0)\in\mathbb{A}_0}(f(p,0)+h_{0,\tau}(p,x))
\]
is a bijection.
\end{theorem}

\begin{proposition}\label{pr3-5}
\[
\{\bar{u}\in C(M\times\mathbb{S}^1,\mathbb{R}^1)\ |\ \exists\ u\in
C(M,\mathbb{R}^1),\
\bar{u}=\lim_{n\to+\infty}U_n^u\}=\mathcal{S}_-.
\]
\end{proposition}

\begin{remark}
Proposition \ref{pr3-5} tells us two things: (i) For each $u\in
C(M,\mathbb{R}^1)$, $\bar{u}=\lim_{n\to+\infty}U_n^u$ is a
backward weak KAM solution of (\ref{1-5}), which proves the second
assertion of Theorem \ref{th1}. (ii) For each $w\in\mathcal{S}_-$
there exists $w_0\in C(M,\mathbb{R}^1)$ such that
$w=\lim_{n\to+\infty}U_n^{w_0}$. Moreover, we know from the proof
of Proposition \ref{pr3-5} that $w_0(x)=w(x,0)$ for all $x\in M$.
\end{remark}

\begin{proof}
First we show that for each $u\in C(M,\mathbb{R}^1)$,
$\bar{u}=\lim_{n\to+\infty}U^u_n$ is a backward weak KAM solution
of (\ref{1-5}). By (\ref{3-18}) we have

\[
\bar{u}(x,\tau)=\inf_{y\in M}\big(u(y)+h_{0,\tau}(y,x)\big)
\]
for all $(x,\tau)\in M\times\mathbb{S}^1$. Combining Lemmas
\ref{le3-2} and \ref{le3-3} we get that $\bar{u}\in\mathcal{S}_-$.

Then we prove that for each $w\in\mathcal{S}_-$, there exists
$w_0\in C(M,\mathbb{R}^1)$ such that
$w=\lim_{n\to+\infty}U^{w_0}_n$. From Theorem \ref{th4} there
exists $f:\mathbb{A}_0\to\mathbb{R}^1$ with $f\prec L$ such that
for each $(x,\tau)\in M\times\mathbb{S}^1$,

\begin{align*}
w(x,\tau) & =\min_{(p,0)\in\mathbb{A}_0}\big(f(p,0)+h_{0,\tau}(p,x)\big)\\
          & =\min_{(p,0)\in\mathbb{A}_0}\Big(f(p,0)+\min_{y\in
          M}\big(h_{0,0}(p,y)+h_{0,\tau}(y,x)\big)\Big)\\
          & =\min_{y\in
          M}\Big(\min_{(p,0)\in\mathbb{A}_0}\big(f(p,0)+h_{0,0}(p,y)\big)+h_{0,\tau}(y,x)\Big)\\
          & =\min_{y\in M}\big(w(y,0)+h_{0,\tau}(y,x)\big).
\end{align*}
Let $w_0(x)=w(x,0)$ for all $x\in M$. Then by Proposition
\ref{pr3-4} and (\ref{3-18}), the uniform limit
$\bar{w}_0=\lim_{n\to+\infty}U^{w_0}_n$ exists and

\[
\bar{w}_0(x,\tau)= \min_{y\in
M}\big(w_0(y)+h_{0,\tau}(y,x)\big)=\min_{y\in
M}\big(w(y,0)+h_{0,\tau}(y,x)\big)
\]
for all $(x,\tau)\in M\times\mathbb{S}^1$. Therefore,
$w=\bar{w}_0=\lim_{n\to+\infty}U^{w_0}_n$.

\end{proof}

\begin{proposition}\label{pr3-6}
Let $u\in C(M\times\mathbb{S}^1,\mathbb{R}^1)$. Then $u$ is a
backward weak KAM solution of (\ref{1-5}) if and only if it is a
viscosity solution of (\ref{1-5}).
\end{proposition}

\begin{proof}
Let $u\in C(M\times\mathbb{S}^1,\mathbb{R}^1)$ and $u_0(x)=u(x,0)$
for all $x\in M$. If $u$ is a backward weak KAM solution of
(\ref{1-5}), then from Proposition \ref{pr3-5} we have
$u=\lim_{n\to+\infty}U^{u_0}_n$. Recall that

\[
U^{u_0}_n(x,\tau)=\tilde{T}^\tau_nu_0(x)=(T_\tau\circ\tilde{T}_nu_0)(x).
\]
It is a standard result that for each $n\in\mathbb{N}$,
$U^{u_0}_n(x,\tau)=(T_\tau\circ\tilde{T}_nu_0)(x)$ is a viscosity
solution of (\ref{1-5}), see \cite{Fat5} for instance. Since $u$
is the uniform limit of $\{U^{u_0}_n\}_{n=1}^{+\infty}$, then from
the stability of viscosity solution of (\ref{1-5}) \cite{Fat-b},
$u$ is also a viscosity solution of (\ref{1-5}).

Suppose now that $u$ is a viscosity solution of (\ref{1-5}). Let
$U^{u_0}(x,t)=T_tu_0(x)$ for all $(x,t)\in M\times[0,+\infty)$.
Then $U^{u_0}$ is a viscosity solution of (\ref{1-5}) with
$U^{u_0}(x,0)=T_0u_0(x)=u_0(x)$. Since $u$ can be considered as a
1-periodic in time viscosity solution on $M\times[0,+\infty)$ and
the Cauchy Problem

\[
\left\{
        \begin{array}{ll}
        v_t+H(x,v_x,t)=0, & \mathrm{on}\ M\times(0,+\infty),\\
        v(x,0)=u_0(x), & \mathrm{on}\ M
                         \end{array}
                         \right.
\]
is well posed in the viscosity sense (see, for example,
\cite{Lio82} or \cite{Ber04}), then
$u(x,t)=U^{u_0}(x,t)=T_tu_0(x)$ for all $(x,t)\in
M\times[0,+\infty)$. Since $u$ is 1-periodic in time, for each
$(x,\tau)\in M\times[0,1]$ we have

\[
u(x,\tau)=u(x,\tau+k)=\inf_\gamma\{u_0(\gamma(0))+\int_0^{\tau+k}L(\gamma,\dot{\gamma},s)ds\}, \quad \forall k\in\mathbb{N},
\]
where the infimum is taken among the
continuous and piecewise $C^1$ paths $\gamma:[0,\tau+k]\to M$ with
$\gamma(\tau+k)=x$. Hence,

\[
u(x,\tau)=\inf_{k\in\mathbb{N} \atop n\leq k\leq
2n}\inf_{y\in M}\big(u_0(y)+F_{0,\tau+k}(y,x)\big)=U^{u_0}_n(x,\tau),\quad \forall n\in\mathbb{N}.
\]
Then by letting $n\to+\infty$, from Proposition \ref{pr3-5} we have
$u=\lim_{n\to+\infty}U^{u_0}_n\in\mathcal{S}_-$.

\end{proof}

\section{The new L-O operator: time-independent case}

As mentioned in the Introduction, in this section we first discuss
the main properties of the new L-O semigroup associated with $L_a$
and then give the proofs of Theorems \ref{th2} and \ref{th3}. Finally,
we construct an example to show that the new L-O semigroup converges
faster than the L-O semigroup in the sense of order when the Aubry set
 $\tilde{\mathcal{A}}_0$ of the Lagrangian system (\ref{1-6}) is a
quasi-periodic invariant torus with Diophantine frequency vector
$\omega\in\mathcal{D}(\rho,\alpha)$.

\subsection{Main properties of the new L-O semigroup}
Let us recall the definition of the new
L-O operator $\tilde{T}^a_t$ associated with $L_a$.
For each $t\geq0$ and each $u\in C(M,\mathbb{R}^1)$,
\[
\tilde{T}^a_tu(x)=\inf_{t\leq \sigma\leq
2t}\inf_{\gamma}\big\{u(\gamma(0))+\int_0^\sigma L_a(\gamma(s),\dot{\gamma}(s))ds\big\}
\]
for all $x\in M$, where the second infimum is taken among the
continuous and piecewise $C^1$ paths $\gamma:[0,\sigma]\rightarrow M$
with $\gamma(\sigma)=x$.

Obviously, $\tilde{T}^a_tu(x)=\inf_{t\leq \sigma\leq 2t}T^a_\sigma u(x)$. Moreover,
it is straightforward to check that for each $t\geq 0$, $\tilde{T}^a_t$ is an operator from
$C(M,\mathbb{R}^1)$ to itself, and that $\{\tilde{T}^a_t\}_{t\geq 0}$ is a semigroup of operators.

\begin{proposition}\label{pr4-1}
For given $t>0$, $u\in C(M,\mathbb{R}^1)$ and $x\in M$, there
exist $\sigma\in[t,2t]$ and a minimizing extremal curve $\gamma:[0,\sigma]\rightarrow
M$ such that $\gamma(\sigma)=x$ and

\[
\tilde{T}^a_tu(x)=u(\gamma(0))+\int_0^\sigma L_a(\gamma,\dot{\gamma})ds.
\]
\end{proposition}

\begin{proof}
Since $\sigma\mapsto T^a_\sigma u(x)$ is continuous on $[t,2t]$ and
$\tilde{T}^a_tu(x)=\inf_{t\leq \sigma\leq 2t}T^a_\sigma u(x)$, then there is
$\sigma_0\in[t,2t]$ such that $\tilde{T}^a_tu(x)=T^a_{\sigma_0}u(x)$. From
the property of the operator $T^a_{\sigma_0}$ (see
\cite [Lemma 4.4.1]{Fat-b}), there exists a minimizing extremal curve
$\gamma:[0,\sigma_0]\rightarrow M$ such that $\gamma(\sigma_0)=x$ and

\[
\tilde{T}^a_tu(x)=T^a_{\sigma_0}u(x)=u(\gamma(0))+\int_0^{\sigma_0}L_a(\gamma,\dot{\gamma})ds.
\]
\end{proof}

Some fundamental properties of $\tilde{T}^a_t$ are discussed in
the following proposition.
\begin{proposition}\label{pr4-2}
\noindent\begin{itemize}
    \item [(1)] For $u$, $v\in C(M,\mathbb{R}^1)$, if $u\leq
              v$, then $\tilde{T}^a_tu\leq\tilde{T}^a_tv$, $\forall t\geq
              0$.
    \item [(2)] If $c$ is a constant and $u\in C(M,\mathbb{R}^1)$,
              then $\tilde{T}^a_t(u+c)=\tilde{T}^a_tu+c$, $\forall t\geq
              0$.
    \item [(3)] For each $u$, $v\in C(M,\mathbb{R}^1)$ and each $t\geq 0$,
              $\|\tilde{T}^a_tu-\tilde{T}^a_tv\|_\infty\leq\|u-v\|_\infty$.
    \item [(4)] For each $u\in C(M,\mathbb{R}^1)$, $\lim_{t\rightarrow
              0^+}\tilde{T}^a_tu=u$.
    \item [(5)] For each $u\in C(M,\mathbb{R}^1)$,
               $(t,x)\mapsto\tilde{T}^a_tu(x)$ is continuous on $[0,+\infty)\times
               M$.
\end{itemize}
\end{proposition}

\begin{remark}
The property (3) means
that the semigroup $\{\tilde{T}^a_t\}_{t\geq 0}$ is continuous at
the origin or of class $C_0$ \cite{Kel}.
\end{remark}

\begin{proof}
Since $T^a_t$ has the monotonicity property (see \cite[Corollary 4.4.4]{Fat-b}),
then

\[
\tilde{T}^a_tu(x)=\inf_{t\leq \sigma\leq 2t}T^a_\sigma u(x)\leq\inf_{t\leq
\sigma\leq 2t}T^a_\sigma v(x)=\tilde{T}^a_tv(x), \quad \forall t>0, \ \forall
x\in M,
\]
i.e., (1) holds. (2) results from the definition of
$\tilde{T}^a_t$ directly. Note that for any $x\in M$,

\[
-\|u-v\|_\infty+v(x)\leq u(x)\leq\|u-v\|_\infty+v(x).
\]
By the properties of $T^a_\sigma$ (see \cite[Corollary 4.4.4]{Fat-b}),
for each $t\geq 0$ we have

\[
T^a_\sigma v(x)-\|u-v\|_\infty\leq T^a_\sigma u(x)\leq
T^a_\sigma v(x)+\|u-v\|_\infty, \quad \forall \sigma\in[t,2t].
\]
Taking the infimum on $\sigma$ over $[t,2t]$ yields

\[
\inf_{t\leq \sigma\leq 2t}T^a_\sigma v(x)-\|u-v\|_\infty\leq\inf_{t\leq \sigma\leq
2t}T^a_\sigma u(x)\leq\inf_{t\leq \sigma\leq 2t}T^a_\sigma v(x)+\|u-v\|_\infty,
\quad \forall x\in M,
\]
and thus (3) holds.

Next we prove (4). For each $u\in C(M,\mathbb{R}^1)$ and each
$\varepsilon>0$, there is $w\in C^1(M,\mathbb{R}^1)$ such that
$\|u-w\|_\infty<\varepsilon$ since $C^1(M,\mathbb{R}^1)$ is a
dense subset of $C(M,\mathbb{R}^1)$ in the topology of uniform
convergence. Thus, we have

\begin{align}\label{4-1}
\begin{split}
\|\tilde{T}^a_tu-u\|_\infty &
\leq\|\tilde{T}^a_tu-\tilde{T}^a_tw\|_\infty
+\|\tilde{T}^a_tw-w\|_\infty+\|w-u\|_\infty\\
&\leq 2\|w-u\|_\infty+\|\tilde{T}^a_tw-w\|_\infty\\
&\leq 2\varepsilon+\|\tilde{T}^a_tw-w\|_\infty, \quad \forall
t\geq 0,
\end{split}
\end{align}
where we have used (3). Since $M$ is compact, then $w$ is
Lipschitz. Denote the Lipschitz constant of $w$ by $K_w$, and by
the superlinearity of $L_a$ there exists $C_{K_w}\in\mathbb{R}^1$
such that

\[
L_a(x,v)\geq K_w\|v\|_x+C_{K_w}, \quad \forall(x,v)\in TM.
\]

For each $x\in M$, each $t\geq 0$ and each continuous and piecewise
$C^1$ path $\gamma:[0,\sigma]\to M$ with $\gamma(\sigma)=x$ and $t\leq \sigma\leq
2t$, since

\[
d(\gamma(0),\gamma(\sigma))\leq\int_0^\sigma\|\dot{\gamma}(s)\|_{\gamma(s)}ds,
\]
then

\[
\int_0^\sigma L_a(\gamma,\dot{\gamma})ds\geq
K_wd(\gamma(0),\gamma(\sigma))+C_{K_w}\sigma\geq
w(\gamma(\sigma))-w(\gamma(0))+C_{K_w}\sigma.
\]
Thus, by the definition of $T^a_\sigma$ we have

\[
T^a_\sigma w(x)\geq w(x)+C_{K_w}\sigma.
\]
Taking the infimum on $\sigma$ over $[t,2t]$ on both sides of this last
inequality yields

\begin{align}\label{4-2}
\tilde{T}^a_tw(x)\geq w(x)+O(t), \quad \mathrm{as}\ t\rightarrow
0^+,
\end{align}
where $O(t)$ is independent of $x$. Using the constant curve
$\gamma_x:[0,\sigma]\rightarrow M$, $s\mapsto x$, we have

\[
T^a_\sigma w(x)\leq w(x)+L_a(x,0)\sigma.
\]
Taking the infimum on $\sigma$ over $[t,2t]$, we obtain

\begin{align}\label{4-3}
\tilde{T}^a_tw(x)\leq w(x)+O(t), \quad \mathrm{as}\ t\rightarrow
0^+,
\end{align}
where $O(t)$ is independent of $x$. Combining (\ref{4-1}),
(\ref{4-2}) and (\ref{4-3}), we have

\[
\lim_{t\rightarrow 0^+}\|\tilde{T}^a_tu-u\|_\infty=0,
\]
i.e., (4) holds.

Finally, we prove (5). For any $(t_0,x_0)\in[0,+\infty)\times M$,
from the semigroup property and (3) we have

\begin{align}\label{4-4}
\begin{split}
|\tilde{T}^a_tu(x)-\tilde{T}^a_{t_0}u(x_0)|&
\leq|\tilde{T}^a_tu(x)-\tilde{T}^a_tu(x_0)|+|\tilde{T}^a_tu(x_0)-\tilde{T}^a_{t_0}u(x_0)|\\
&\leq
|\tilde{T}^a_tu(x)-\tilde{T}^a_tu(x_0)|+\|\tilde{T}^a_tu-\tilde{T}^a_{t_0}u\|_\infty\\
&\leq|\tilde{T}^a_tu(x)-\tilde{T}^a_tu(x_0)|+\|\tilde{T}^a_{|t-t_0|}u-u\|_\infty.
\end{split}
\end{align}
From (\ref{4-4}), $\tilde{T}^a_tu\in C(M,\mathbb{R}^1)$ and
(4), we conclude that (5) holds.
\end{proof}

The proposition below establishs a relationship between
$\tilde{T}^a_t$ and $T^a_t$.

\begin{proposition}\label{pr4-3}
\noindent
\begin{itemize}
    \item [(1)] For each $u\in C(M,\mathbb{R}^1)$, the uniform limit
              $\lim_{t\rightarrow+\infty}\tilde{T}^a_tu$ exists and
              \[
              \lim_{t\rightarrow+\infty}\tilde{T}^a_tu=\lim_{t\rightarrow+\infty}T^a_tu=\bar{u}.
              \]
    \item [(2)] For each $t\geq 0$ and each $u\in C(M,\mathbb{R}^1)$,
              $\|\tilde{T}^a_tu-\bar{u}\|_\infty\leq\|T^a_tu-\bar{u}\|_\infty.$
    \item [(3)] $u\in C(M,\mathbb{R}^1)$ is a fixed point of $\{\tilde{T}^a_t\}_{t\geq
              0}$ if and only if it is a fixed point of $\{T^a_t\}_{t\geq0}.$
\end{itemize}
\end{proposition}

\begin{remark}
From (1) $\lim_{t\rightarrow+\infty}\tilde{T}^a_tu$ exists and is
a backward weak KAM solution of the Hamilton-Jacobi equation
$H_a(x,u_x)=0$. (2) essentially says that the new L-O semigroup
converges faster than the L-O semigroup. (3) implies that $u\in
C(M,\mathbb{R}^1)$ is a backward weak KAM solution if and only if
it is a fixed point of $\{\tilde{T}^a_t\}_{t\geq 0}$.
\end{remark}

\begin{remark}\label{re4-1}
Just as we mentioned earlier, for each $\tau\in[0,1]$ and each $u\in
C(M,\mathbb{R}^1)$, the uniform limit
$\lim_{n\to+\infty}\tilde{T}_n^{a,\tau}u$ exists and

\[
\lim_{n\to+\infty}\tilde{T}_n^{a,\tau}u=\lim_{n\to+\infty}T_n^au=\bar{u}.
\]
It can be proved by slight modifications of the proof of (1)
in Proposition \ref{pr4-3}.
\end{remark}

\begin{proof}
First we prove (1). Assume by contradiction that there exist
$\varepsilon_0>0$, $t_n\rightarrow+\infty$ and $x_n\in M$ such
that

\[
|\tilde{T}^a_{t_n}u(x_n)-\bar{u}(x_n)|\geq\varepsilon_0.
\]
From the compactness of $M$, without loss of generality we assume
that $x_n\rightarrow x_0$, $n\rightarrow+\infty$. In view of the
definition of $\tilde{T}^a_t$, there exist $\sigma_n\in[t_n,2t_n]$ such
that

\[
|T^a_{\sigma_n}u(x_n)-\bar{u}(x_n)|\geq\varepsilon_0.
\]
Let $n\rightarrow+\infty$. Since $(\sigma,x)\mapsto T^a_\sigma u(x)$ is
continuous, then we have

\[
\lim_{\sigma\rightarrow+\infty}T^a_\sigma u(x_0)\neq\bar{u}(x_0),
\]
which contradicts $\lim_{\sigma\rightarrow+\infty}T^a_\sigma u=\bar{u}$.

Next we show (2). For each $t\geq 0$ and each $x\in M$, there
exists $t\leq \sigma_x\leq 2t$ such that

\[
|\tilde{T}^a_tu(x)-\bar{u}(x)|=|T^a_{\sigma_x}u(x)-\bar{u}(x)|.
\]
Since $\bar{u}$ is a fixed point of $\{T^a_t\}_{t\geq 0}$, then we
have
$|T^a_{\sigma_x}u(x)-\bar{u}(x)|=|T^a_{\sigma_x}u(x)-T^a_{\sigma_x}\bar{u}(x)|\leq
\|T^a_{\sigma_x}u-T^a_{\sigma_x}\bar{u}\|_\infty=\|T^a_{\sigma_x-t}\circ
T^a_tu-T^a_{\sigma_x-t}\circ T^a_t\bar{u}\|_\infty\leq
\|T^a_tu-T^a_t\bar{u}\|_\infty=\|T^a_tu-\bar{u}\|_\infty$, where
we have used the non-expansiveness property of $T^a_{\sigma_x-t}$
(see \cite[Corollary 4.4.4]{Fat-b}). Hence (2) holds.

At last, we show (3). Suppose that $u$ is a fixed point of
$\{T^a_t\}_{t\geq 0}$, i.e., $T^a_tu=u$, $\forall t\geq 0$. Then
$\lim_{t\rightarrow+\infty}T^a_tu=u$. From (2) we have

\[
\|\tilde{T}^a_tu-u\|_\infty\leq\|T^a_tu-u\|_\infty=0, \quad
\forall t\geq 0,
\]
which implies that $u$ is a fixed point of
$\{\tilde{T}^a_t\}_{t\geq 0}$. Suppose conversely that $u$ is a
fixed point of $\{\tilde{T}^a_t\}_{t\geq 0}$. Then from (1)
$\lim_{t\rightarrow+\infty}\tilde{T}^a_tu=u=\lim_{t\rightarrow+\infty}T^a_tu$.
Hence $u$ is a backward weak KAM solution of $H_a(x,u_x)=0$ and a
fixed point of $\{T^a_t\}_{t\geq 0}$.
\end{proof}

\subsection{Rates of convergence of the L-O semigroup and the new L-O
semigroup} Recall the $C^2$ positive definite and superlinear
Lagrangian (\ref{1-6})

\begin{align*}
L^1_a(x,v)=\frac{1}{2}\langle
A(x)(v-\omega),(v-\omega)\rangle+f(x,v-\omega), \quad x\in
\mathbb{T}^n,\ v\in\mathbb{R}^n.
\end{align*}
The conjugated Hamiltonian
$H^1_a:\mathbb{T}^n\times\mathbb{R}^n\rightarrow\mathbb{R}^1$ of
$L_a^1$ has the following form

\[
H^1_a(x,p)=\langle\omega,p\rangle+\frac{1}{2}\langle
A^{-1}(x)p,p\rangle+g(x,p),
\]
where $g(x,p)=O(\|p\|^3)$ as $p\rightarrow 0$. It is clear that
$H^1_a(x,0)=0$ and thus $w\equiv const.$ is a smooth viscosity
solution of the corresponding Hamilton-Jacobi equation
$H^1_a(x,u_x)=0$. In view of the Legendre transform,

\[
L^1_a(x,v)=L^1_a(x,v)-\langle
w_x,v\rangle\geq-H^1_a(x,w_x)=-H^1_a(x,0)=0, \quad \forall
(x,v)\in\mathbb{T}^n\times\mathbb{R}^n.
\]
Furthermore, if
$(x,v)\in\tilde{\mathcal{M}}_0=\cup_{x\in\mathbb{T}^n}(x,\omega)$,
then $w_x=\frac{\partial L}{\partial v}(x,v)$ (see \cite[Theorem 4.8.3]{Fat-b}),
from which we have

\[
L^1_a(x,v)=L^1_a(x,v)-\langle
w_x,v\rangle=-H^1_a(x,w_x)=-H^1_a(x,0)=0.
\]
Hence

\[
L^1_a\geq 0, \quad \forall (x,v)\in\mathbb{T}^n\times\mathbb{R}^n
\]
and in particular,

\[
L^1_a|_{\cup_{x\in\mathbb{T}^n}(x,\omega)}=0.
\]

For each $u\in C(\mathbb{T}^n,\mathbb{R}^1)$, because of
$c(L^1_a)=0$ we have $\lim_{t\rightarrow+\infty}T_t^au=\bar{u}$.
Note that both $w\equiv const.$ and $\bar{u}$ are  viscosity
solutions of $H^1_a(x,u_x)=0$. Hence $\bar{u}\equiv const.$ since
the viscosity solution of $H^1_a(x,u_x)=0$ is unique up to
constants when $\mathcal{A}_0=\mathbb{T}^n$ \cite{Lia}, where
$\mathcal{A}_0$ is the projected Aubry set.

\subsubsection{Rate of convergence of the L-O semigroup}
We present here the proof of Theorem \ref{th2}. For this, the
following lemma is needed.

\begin{lemma}\label{le4-1}
For each $u\in C(\mathbb{T}^n,\mathbb{R}^1)$,
$\bar{u}\equiv\min_{x\in\mathbb{T}^n}u(x)$.
\end{lemma}

\begin{proof}
For any $x\in\mathbb{T}^n$, from the definition of $T^a_t$ we have

\[
\bar{u}(x)=\lim_{t\rightarrow+\infty}T^a_tu(x)
=\lim_{t\rightarrow+\infty}\inf_{z\in\mathbb{T}^n}
\{u(z)+\int_0^tL^1_a(\gamma_z,\dot{\gamma}_z)ds\},
\]
where $\gamma_z:[0,t]\rightarrow\mathbb{T}^n$ is a Tonelli
minimizer with $\gamma_z(0)=z$, $\gamma_z(t)=x$. Since $L^1_a\geq
0$, then $\bar{u}(x)\geq\min_{z\in\mathbb{T}^n}u(z)$ and therefore
it suffices to show that
$\bar{u}(x)\leq\min_{z\in\mathbb{T}^n}u(z)$.

Take $y\in\mathbb{T}^n$ with $u(y)=\min_{z\in\mathbb{T}^n}u(z)$.
Consider the following two curves

\[
\gamma_\omega:[0,t]\rightarrow\mathbb{T}^n,\  s\mapsto \omega s+y
\]
and

\[
\gamma_{\omega'}:[0,t]\rightarrow\mathbb{T}^n,\  s\mapsto \omega'
s+y
\]
with $\gamma_{\omega'}(t)=x$, where $\omega'\in\mathbb{S}^{n-1}$
and $t>0$. It is clear that $\gamma_{\omega'}$ is a curve in
$\mathbb{T}^n$ connecting $y$ and $x$. Let
$\Delta=\gamma_{\omega'}(t)-\gamma_\omega(t)=x-(\omega t+y)$. Then
$\|\Delta\|\leq\frac{\sqrt{n}}{2}$ and
$\dot{\gamma}_{\omega'}\equiv\omega'=\frac{\Delta}{t}+\omega$.
Therefore, we have

\begin{align*}
T^a_tu(x) & \leq
          u(\gamma_{\omega'}(0))+\int_0^tL^1_a(\gamma_{\omega'},\dot{\gamma}_{\omega'})ds\\
        & = u(y)+\int_0^t\Big(\frac{1}{2}\langle
        A(\gamma_{\omega'})(\omega'-\omega),(\omega'-\omega)\rangle+f(\gamma_{\omega'},\omega'-\omega)\Big)ds\\
        & = u(y)+\int_0^t\Big(\frac{1}{2}\Big\langle
        A(\gamma_{\omega'})\frac{\Delta}{t},\frac{\Delta}{t}\Big\rangle+f(\gamma_{\omega'},\frac{\Delta}{t})\Big)ds\\
        & \leq u(y)+\frac{C}{t}+O(\frac{1}{t^2}),
\end{align*}
where $C$ is a constant, which depends only on $n$.

From the arguments above we know that for any $\varepsilon>0$,
there exists $T>0$ such that for any $t>T$ there exists
$\gamma_{\omega'}:[0,t]\rightarrow\mathbb{T}^n$ with
$\gamma_{\omega'}(t)=x$, and

\[
T^a_tu(x)\leq
u(\gamma_{\omega'}(0))+\int_0^tL^1_a(\gamma_{\omega'},\dot{\gamma}_{\omega'})ds\leq
\min_{z\in\mathbb{T}^n}u(z)+\varepsilon.
\]
Hence
$\bar{u}(x)=\lim_{t\rightarrow+\infty}T^a_tu(x)\leq\min_{z\in\mathbb{T}^n}u(z)$.
\end{proof}

\noindent\emph{Proof of Theorem \ref{th2}.} In order to prove our
result, it is sufficient to show that for each $u\in
C(\mathbb{T}^n,\mathbb{R}^1)$, there exists a constant $K>0$ such
that the following two inequalities hold.

\[
T^a_tu(x)-\bar{u}(x)\leq \frac{K}{t}, \quad \forall t>0,\ \forall
x\in\mathbb{T}^n; \eqno (\mathrm{I1})
\]

\[
\bar{u}(x)-T^a_tu(x)\leq \frac{K}{t}, \quad \forall t>0,\ \forall
x\in\mathbb{T}^n. \eqno (\mathrm{I2})
\]
Obviously, (I2) holds. In fact, for each $t>0$ and each
$x\in\mathbb{T}^n$, from the definition of $T^a_t$ we have

\[
T^a_tu(x) =\inf_{z\in\mathbb{T}^n}
\{u(z)+\int_0^tL^1_a(\gamma_z,\dot{\gamma}_z)ds\},
\]
where $\gamma_z:[0,t]\rightarrow\mathbb{T}^n$ is a Tonelli
minimizer with $\gamma_z(0)=z$, $\gamma_z(t)=x$. In view of
$L^1_a\geq 0$ and Lemma \ref{le4-1}, we have

\[
T^a_tu(x) =\inf_{z\in\mathbb{T}^n}
\{u(z)+\int_0^tL^1_a(\gamma_z,\dot{\gamma}_z)ds\}\geq\min_{z\in\mathbb{T}^n}u(z)=\bar{u}(x).
\]
Thus $\bar{u}(x)-T^a_tu(x)\leq 0$, $\forall t>0,\ \forall
x\in\mathbb{T}^n$ and (I2) holds.

Next we prove (I1). It suffices to show that there exists a
constant $C>0$ such that for sufficiently large $t>0$,

\begin{align}\label{4-5}
T^a_tu(x)-\bar{u}(x)\leq \frac{C}{t}, \quad \forall
x\in\mathbb{T}^n,
\end{align}
where $C$ depends only on $n$. In deed, since $(s,z)\mapsto
T_su(z)$ is continuous on $[0,\infty)\times\mathbb{T}^n$, if
(\ref{4-5}) holds, then there exists a constant $K>0$ such that

\[
T^a_tu(x)-\bar{u}(x)\leq\frac{K}{t}, \quad \forall t>0,\ \forall
x\in\mathbb{T}^n,
\]
where $K$ depends only on $n$ and $u$.

Take $y\in\mathbb{T}^n$ with $u(y)=\min_{z\in\mathbb{T}^n}u(z)$.
Let us consider the following curve in $\mathbb{T}^n$

\[
\gamma_\omega:[0,t]\rightarrow\mathbb{T}^n,\  s\mapsto \omega s+y,
\]
where $t>0$. Then for each $x\in\mathbb{T}^n$, let

\[
\gamma_{\omega'}:[0,t]\rightarrow\mathbb{T}^n,\  s\mapsto \omega'
s+y
\]
be a curve in $\mathbb{T}^n$ connecting $y$ and $x$, where
$\omega'\in\mathbb{S}^{n-1}$. Let
$\Delta=\gamma_{\omega'}(t)-\gamma_\omega(t)=x-(\omega t+y)$. Then
$\|\Delta\|\leq\frac{\sqrt{n}}{2}$ and
$\dot{\gamma}_{\omega'}\equiv\omega'=\frac{\Delta}{t}+\omega$.
Hence,

\begin{align*}
T^a_tu(x) & \leq
          u(\gamma_{\omega'}(0))+\int_0^tL^1_a(\gamma_{\omega'},\dot{\gamma}_{\omega'})ds\\
        & = u(y)+\int_0^t\Big(\frac{1}{2}\langle
        A(\gamma_{\omega'})(\omega'-\omega),(\omega'-\omega)\rangle+f(\gamma_{\omega'},\omega'-\omega)\Big)ds\\
        & = u(y)+\int_0^t\Big(\frac{1}{2}\Big\langle
        A(\gamma_{\omega'})\frac{\Delta}{t},\frac{\Delta}{t}\Big\rangle+f(\gamma_{\omega'},\frac{\Delta}{t})\Big)ds\\
        & \leq u(y)+\frac{C_1}{t}+O(\frac{1}{t^2}),
\end{align*}
where $C_1$ is a constant which depends only on $n$. From Lemma
\ref{le4-1}, we have $T^a_tu(x)-\bar{u}(x)\leq\frac{C}{t}$ for
$t>0$ large enough, where $C$ is a constant which still depends
only on $n$, i.e., (\ref{4-5}) holds. \hfill $\Box$

\subsubsection{Rate of convergence of the new L-O semigroup}
To complete the proof of Theorem \ref{th3}, we review
preliminaries on the ergodization rate for linear flows on the
torus $\mathbb{T}^n$, i.e., the rate at which the image of a point
fills the torus when subjected to linear flows. There is a direct
relationship between the rate of convergence of the new L-O
semigroup and the ergodization rate for linear flows on the torus
$\mathbb{T}^n$. Let us recall the following result of Dumas'
\cite{Dum} concerning the estimate of ergodization time.

For each $t\in\mathbb{R}^1$ and each $\omega\in\mathbb{S}^{n-1}$,
consider the one-parameter family of translation maps
$\omega_t:\mathbb{T}^n\rightarrow\mathbb{T}^n$, $x\mapsto x+\omega
t$. A rectilinear orbit of $\mathbb{T}^n$ with direction vector
$\omega$ and initial condition $x$ is defined as the image of $x$
under the linear flow $\omega_t$ over some closed interval
$[t_0,t_1]\subset\mathbb{R}^1$, i.e.,

\[
\bigcup_{t_0\leq t\leq t_1}\omega_t(x).
\]

Given $R>0$, the direction vector $\omega\in\mathbb{S}^{n-1}$ is
said to ergodize $\mathbb{T}^n$ to within $R$ after time $T$ if

\begin{align}\label{4-6}
\bigcup_{0\leq t\leq T}\omega_t(B_R(x))=\mathbb{T}^n
\end{align}
for all $x\in\mathbb{T}^n$.

As defined in the Introduction, for $\rho>n-1$ and $\alpha>0$,

\[
\mathcal{D}(\rho,\alpha)=\Big\{\beta\in \mathbb{S}^{n-1}|\
|\langle\beta,k\rangle|>\frac{\alpha}{|k|^\rho},\ \forall
k\in\mathbb{Z}^n\backslash\{0\}\Big\},
\]
whose elements can not be approximated by rationals too rapidly.

\begin{theorem}[Dumas \cite{Dum}]\label{th5}
Let $0<R\leq 1$. Given any highly nonresonant direction vector
$\omega\in\mathcal{D}(\rho,\alpha)$, rectilinear orbits of
$\mathbb{T}^n$ with direction vector $\omega$ will ergodize
$\mathbb{T}^n$ to within $R$ after time T, where

\[
T=\frac{2\|V_*\|_\triangle}{\alpha\pi R^{\rho+n/2}}
\]
is independent of $\omega$.
\end{theorem}

\begin{remark}
The constant $\|V_*\|_\triangle$ is a Sobolev norm of a certain
``smoothest test function" and it depends only on $n$ and $\rho$.
See \cite{Dum} for complete details.
\end{remark}

We are now in a position to give the proof of Theorem \ref{th3}.

\noindent\emph{Proof of Theorem \ref{th3}.} Our purpose is to show
that for each $u\in C(\mathbb{T}^n,\mathbb{R}^1)$, there exists a
constant $\tilde{K}>0$ such that the following two inequalities
hold.

\[
\tilde{T}^a_tu(x)-\bar{u}(x)\leq
\tilde{K}t^{-(1+\frac{4}{2\rho+n})}, \quad \forall t>0,\ \forall
x\in\mathbb{T}^n; \eqno (\mathrm{I3})
\]

\[
\bar{u}(x)-\tilde{T}^a_tu(x)\leq
\tilde{K}t^{-(1+\frac{4}{2\rho+n})}, \quad \forall t>0,\ \forall
x\in\mathbb{T}^n. \eqno (\mathrm{I4})
\]

First we show (I4). For each $t>0$ and each $x\in\mathbb{T}^n$, by
the definition of $\tilde{T}^a_t$ we have

\[
\tilde{T}^a_tu(x) =\inf_{t\leq \sigma\leq
2t}\inf_{z\in\mathbb{T}^n} \{u(z)+\int_0^\sigma
L^1_a(\gamma_z,\dot{\gamma}_z)ds\},
\]
where $\gamma_z:[0,\sigma]\rightarrow\mathbb{T}^n$ is a Tonelli
minimizer with $\gamma_z(0)=z$, $\gamma_z(\sigma)=x$. In view of
$L^1_a\geq 0$ and Lemma \ref{le4-1}, we have

\[
\tilde{T}^a_tu(x) =\inf_{t\leq \sigma\leq
2t}\inf_{z\in\mathbb{T}^n} \{u(z)+\int_0^\sigma
L^1_a(\gamma_z,\dot{\gamma}_z)ds\}\geq\min_{z\in\mathbb{T}^n}u(z)=\bar{u}(x).
\]
Thus $\bar{u}(x)-\tilde{T}^a_tu(x)\leq 0$, $\forall t>0,\ \forall
x\in\mathbb{T}^n$, i.e., (I4) holds.

Then it remains to show (I3). When $R=1$, according to Theorem
\ref{th5} the ergodization time
$T=\frac{2\|V_*\|_\triangle}{\alpha\pi}$. For any $t\geq T$, let
$R_t=\sqrt[\rho+n/2]{\frac{2\|V_*\|_\triangle}{\alpha\pi t}}$.
Then $0<R_t\leq 1$.

Take $y\in\mathbb{T}^n$ with $u(y)=\min_{z\in\mathbb{T}^n}u(z)$.
Let $y_t=\omega_t(y)=\omega t+y$. For $R_t$ defined above, since
$\omega\in\mathcal{D}(\rho,\alpha)$, then from Theorem \ref{th5}
and (\ref{4-6}) we have

\[
\bigcup_{0\leq\varsigma\leq
t}\omega_\varsigma(B_{R_t}(y_t))=\mathbb{T}^n.
\]

Therefore, for each $x\in\mathbb{T}^n$, there exists
$0\leq\varsigma'\leq t$ such that
$d_{\mathbb{T}^n}(\omega_{\varsigma'}(y_t),x)\leq R_t$, i.e.,
$d_{\mathbb{T}^n}(\omega(t+\varsigma')+y,x)\leq R_t$. Equivalently
this means that there exists $t\leq \sigma'\leq 2t$ such that

\[
d_{\mathbb{T}^n}(\omega \sigma'+y,x)\leq R_t,
\]
where $\sigma'=t+\varsigma'$. Consider the following curve in
$\mathbb{T}^n$

\[
\gamma_{\omega'}:[0,\sigma']\rightarrow\mathbb{T}^n,\  s\mapsto
\omega's+y
\]
with $\gamma_{\omega'}(\sigma')=x$, where
$\omega'\in\mathbb{S}^{n-1}$. It is clear that $\gamma_{\omega'}$
connects $y$ and $x$. Let
$\Delta=\gamma_{\omega'}(\sigma')-\omega_{\sigma'}(y)=x-(\omega
\sigma'+y)$. Then $\|\Delta\|=d_{\mathbb{T}^n}(x,\omega
\sigma'+y)\leq R_t$ and
$\dot{\gamma}_{\omega'}\equiv\omega'=\frac{\Delta}{\sigma'}+\omega$.
Hence we have

\begin{align*}
\tilde{T}^a_tu(x)-\bar{u}(x) & \leq u(\gamma_{\omega'}(0))+\int_0^{\sigma'}L^1_a(\gamma_{\omega'},\dot{\gamma}_{\omega'})ds-\bar{u}(x)\\
                           & =\int_0^{\sigma'}\Big(\frac{1}{2}\langle
                           A(\gamma_{\omega'})(\omega'-\omega),(\omega'-\omega)\rangle+f(\gamma_{\omega'},\omega'-\omega)\Big)ds\\
                           &\leq \frac{CR_t^2}{t}
\end{align*}
for sufficiently large $t>0$ and some constant $C>0$. Since
$R_t^2=(\frac{2\|V_*\|_\triangle}{\alpha\pi
t})^{\frac{2}{\rho+n/2}}$, then for $t>0$ large enough we have

\[
\tilde{T}^a_tu(x)-\bar{u}(x)\leq C_1t^{-(1+\frac{4}{2\rho+n})},
\quad \forall x\in\mathbb{T}^n,
\]
where $C_1$ is a constant which depends only on $n$, $\rho$ and
$\alpha$. From (5) of Proposition \ref{pr4-2}, $(s,z)\mapsto
\tilde{T}^a_s u(z)$ is continuous on
$[0,\infty)\times\mathbb{T}^n$. Hence there exists a constant
$\tilde{K}>0$ such that

\[
\tilde{T}^a_tu(x)-\bar{u}(x)\leq
\tilde{K}t^{-(1+\frac{4}{2\rho+n})}, \quad \forall t>0,\ \forall
x\in\mathbb{T}^n,
\]
where $\tilde{K}$ depends only on $n$, $\rho$, $\alpha$ and $u$,
i.e., (I3) holds. \hfill $\Box$

\subsubsection{An example}
\begin{example}\label{ex1}
Consider the following integrable $C^2$ Lagrangian

\[
\bar{L}^1_a(x,v)=\frac{1}{2}\langle v-\omega,v-\omega\rangle,
\quad x\in\mathbb{T}^n,\ v\in\mathbb{R}^n,\
\omega\in\mathbb{S}^{n-1}.
\]
\end{example}
It is easy to see that $\bar{L}^1_a$ is a special case of $L^1_a$.
For $\bar{L}^1_a$, we show that there exist $u\in
C(\mathbb{T}^n,\mathbb{R}^1)$, $x^0\in\mathbb{T}^n$ and
$t_m\to+\infty$ as $m\to+\infty$ such that

\[
|T^a_{t_m}u(x^0)-\bar{u}(x^0)|=O(\frac{1}{t_m}), \quad
m\to+\infty,
\]
which implies that the result of Theorem \ref{th2} is sharp in the
sense of order.

Recall the universal covering projection
$\pi:\mathbb{R}^n\to\mathbb{T}^n$. Let $x^0\in\mathbb{T}^n$ such
that each point $\tilde{x}^0\in\mathbb{R}^n$ in the fiber over
$x^0$ ($\pi\tilde{x}^0=x^0$) is the center of each fundamental
domain in $\mathbb{R}^n$. Define a continuous function on
$\mathbb{R}^n$ as follows: for $\tilde{x}\in\mathbb{R}^n$

\[
\tilde{u}(\tilde{x})=\left\{
                         \begin{array}{ll}
                         \delta-\|\tilde{x}-\tilde{x}^0\|, &
                         \|\tilde{x}-\tilde{x}^0\|\leq\delta,\\
                         0, & \mathrm{otherwise},
                         \end{array}
                         \right.
\]
where $0<\delta<\frac{1}{2}$. We then define a continuous function
on $\mathbb{T}^n$ as $u(x)=\tilde{u}(\tilde{x})$ for all
$x\in\mathbb{T}^n$, where $\tilde{x}$ is an arbitrary point in the
fiber over $x$. Thus, from Lemma \ref{le4-1},
$\bar{u}\equiv\min_{x\in\mathbb{T}^n}u(x)=0$.

Now fix a point $\tilde{x}^0_0$ in the fiber over $x^0$. Then
there exist $\{\tilde{x}^0_m\}_{m=1}^{+\infty}$ in the fiber over
$x^0$ and $t_m\to+\infty$ as $m\to+\infty$ such that
$\|(\tilde{x}^0_m-\omega
t_m)-\tilde{x}^0_0\|\leq\frac{\delta}{2}$. Let
$\tilde{z}_m=\tilde{x}^0_m-\omega t_m$. Then
$\|\tilde{z}_m-\tilde{x}^0_0\|\leq\frac{\delta}{2}$. For each
$t_m$ there exists $y_m\in\mathbb{T}^n$ such that

\[
T^a_{t_m}u(x^0)=u(y_m)+\int_0^{t_m}\bar{L}^1_a(\gamma_{y_m},\dot{\gamma}_{y_m})ds,
\]
where $\gamma_{y_m}:[0,t_m]\to\mathbb{T}^n$ is a Tonelli minimizer
with $\gamma_{y_m}(0)=y_m$, $\gamma_{y_m}(t_m)=x^0$. In view of
the lifting property of the covering projection, there is a unique
curve $\tilde{\gamma}_{y_m}:[0,t_m]\to\mathbb{R}^n$ with
$\pi\tilde{\gamma}_{y_m}=\gamma_{y_m}$ and
$\tilde{\gamma}_{y_m}(t_m)=\tilde{x}^0_m$. Set
$\tilde{y}_m=\tilde{\gamma}_{y_m}(0)$. Then $\pi\tilde{y}_m=y_m$.
Moreover, $\tilde{\gamma}_{y_m}$ has the following form

\[
\tilde{\gamma}_{y_m}(s)=\omega's+\tilde{y}_m, \quad s\in[0,t_m],
\]
where $\omega'\in\mathbb{S}^{n-1}$. It is clear that
$\tilde{\gamma}_{y_m}(0)=\tilde{y}_m$ and
$\tilde{y}_m=\tilde{x}^0_m-\omega't_m$.

If $\|\tilde{y}_m-\tilde{z}_m\|\leq\frac{\delta}{4}$, then from
$\|\tilde{z}_m-\tilde{x}^0_0\|\leq\frac{\delta}{2}$ we have
$\|\tilde{y}_m-\tilde{x}^0_0\|\leq\frac{3\delta}{4}$. Hence,

\begin{align}\label{4-7}
\begin{split}
T^a_{t_m}u(x^0)&=u(y_m)+\int_0^{t_m}\bar{L}^1_a(\gamma_{y_m},\dot{\gamma}_{y_m})ds\\
               &\geq\tilde{u}(\tilde{y}_m)\geq\delta-\frac{3\delta}{4}=\frac{\delta}{4}.
\end{split}
\end{align}
From (\ref{4-7}), we may deduce that there can only be a  finite
number of  $\tilde{y}_m$'s such that
$\|\tilde{y}_m-\tilde{z}_m\|\leq\frac{\delta}{4}$. For, otherwise,
there would be $\{t_{m_i}\}_{i=1}^{+\infty}$ and
$\{\tilde{y}_{m_i}\}_{i=1}^{+\infty}$ such that

\[
T^a_{t_{m_i}}u(x^0)\geq\frac{\delta}{4}, \quad i=1,2,\cdots,
\]
which contradicts
$\lim_{i\to+\infty}T^a_{t_{m_i}}u(x^0)=\bar{u}(x^0)=0$.

For $\tilde{y}_m$ with
$\|\tilde{y}_m-\tilde{z}_m\|>\frac{\delta}{4}$, we have

\[
\frac{\delta}{4}<\|\tilde{y}_m-\tilde{z}_m\|=\|\tilde{x}_m^0-\omega't_m-(\tilde{x}_m^0-\omega
t_m)\|=\|\omega-\omega'\|t_m.
\]
Thus,

\begin{align}\label{4-8}
\begin{split}
T^a_{t_m}u(x^0)&=u(y_m)+\int_0^{t_m}\bar{L}^1_a(\gamma_{y_m},\dot{\gamma}_{y_m})ds\\
               &\geq\frac{1}{2}t_m\|\omega-\omega'\|^2=\frac{1}{2}\frac{t^2_m\|\omega-\omega'\|^2}{t_m}\geq\frac{\delta^2}{32t_m}.
\end{split}
\end{align}
Therefore, from (\ref{4-8}) and Theorem \ref{th2} we have

\[
|T^a_{t_m}u(x^0)-\bar{u}(x^0)|=|T^a_{t_m}u(x^0)|=O(\frac{1}{t_m}),
\quad m\to+\infty.
\]


\end{document}